\documentclass{article}
\usepackage{float}
\usepackage{cite}
\usepackage{shapepar}
\usepackage{listings}
\usepackage{amsmath}
\usepackage{amsthm}
\usepackage{mathtools}
\usepackage{graphics}
\usepackage{amssymb}
\usepackage{makecell}
\usepackage{mathrsfs}
\usepackage{framed}
\usepackage{diagbox}
\usepackage{booktabs}
\usepackage{fancybox}
\usepackage{authblk}
\usepackage{geometry}
\usepackage{multirow}
\usepackage{enumerate}
\usepackage{subcaption}
\usepackage{caption}
\usepackage{hyperref}
\usepackage{boldline}
\usepackage[table]{xcolor}
\usepackage{slashbox}
\usepackage{tabularray}
\usepackage{vcell}
\usepackage{titlesec}
\usepackage{graphicx} 
\usepackage{array}
\numberwithin{equation}{section}
\newtheorem{theorem}{Theorem}[section]
\usepackage[ruled,linesnumbered]{algorithm2e}
\usepackage{enumitem}
\newtheorem{lemma}{Lemma}[section]
\newtheorem{remark}{Remark}[section]

\newcommand\keywords[1]{\textbf{Keywords}: #1}

\newtheorem{condition}{Condition}[section]

\title{Constrained Neural Parameterization for Optimization\\ in Function Spaces}

\author[$\dagger \ddagger$]{Michael Hinterm\"uller\textsuperscript{1}}
\author[$\dagger$]{Jianfeng Ning\textsuperscript{2}}

\affil[$\dagger$]{\small Weierstrass Institute for Applied Analysis and Stochastics, Berlin, Germany}
\affil[$\ddagger$]{\small Institute for Mathematics, Humboldt-Universit\"at zu Berlin, Berlin, Germany}

\date{}

\begin{document}

\maketitle
\footnotetext[1]{Email: \texttt{hintermueller@wias-berlin.de}}
\footnotetext[2]{Email: \texttt{ning@wias-berlin.de}}
\setcounter{footnote}{2}

\begin{abstract}
%Constrained optimization in function spaces arises in a wide range of applications.   In this work, we 
We propose constrained neural parameterization schemes for several classes of constraints arising in optimization problems in function spaces.
This is achieved by constructing smooth neural parameterizations whose image lies entirely in the admissible set while remaining asymptotically dense. In this way, the original constrained optimization problem is transformed into a smooth unconstrained problem in parameter space, enabling efficient gradient-based optimization without penalty parameters or Lagrange multipliers. We develop geometric constructions for polyhedral constraint sets in Hilbert spaces, propose smooth neural architectures for some pointwise constraints, and introduce an exact reduced neural formulation for PDE constraints that admit a separable structure. 
Numerical experiments demonstrate the effectiveness of the proposed methods.
\end{abstract}

\section{Introduction}

Constrained optimization in function spaces plays a central role in many areas of applied mathematics, including optimal control, inverse problems, and shape optimization. In these applications one often seeks to minimize a functional over an admissible set defined by structural conditions such as pointwise inequalities, integral constraints, or partial differential equations (PDEs). Such constraints are not merely technical requirements but typically reflect intrinsic modeling principles: pointwise bounds encode physical limitations, integral constraints represent conservation laws or resource restrictions, complementarity conditions arise from equilibrium phenomena, and PDE constraints describe governing physical laws. Designing numerical methods that are both computationally efficient and capable of reliably enforcing such constraints therefore is of paramount importance.

Consider the constrained optimization problem in a Banach space $X$ equipped with the norm $\|\cdot\|_X$:
\begin{equation}\label{eq:gen_opt}
\min_{x\in X} J(x) 
\quad \text{subject to} \quad 
x\in\mathcal K\subset X,
\end{equation}
where $J:X\to\mathbb{R}$ is the objective functional and $\mathcal K$ denotes an admissible set. In this work we propose a class of \emph{constrained neural parameterization} (CNP) methods that enforce feasibility by construction for some common constraints. The central idea, advocated here, is to introduce a smooth parametric mapping
\begin{equation}
    \mathcal T_M:\mathbb R^M\to X,
\end{equation}
whose image is contained entirely in $\mathcal K$. The mapping is designed so that its range becomes dense in $\mathcal K$ as the parameter dimension increases. This means in particular that $\forall\,x\in \mathcal{K}$, there exists a sequence $\{\theta_M\}$ such that
\begin{equation}
 \|\mathcal{T}_M(\theta_M)-x\|_X\rightarrow0,\quad\text{as }M\rightarrow \infty.
\end{equation}
Consequently, the constrained problem \eqref{eq:gen_opt} can be approximated by the unconstrained finite-dimensional optimization problem for sufficiently large $M$:
\begin{equation}\label{eq:unconstrained}
\min_{\theta\in\mathbb R^M} J(\mathcal T_M(\theta)),
\end{equation}
whose candidate solutions automatically satisfy the constraints. Gradient-based optimization can therefore be performed directly in parameter space without the need for penalty parameters, projections, or multiplier updates. Thus, feasibility is enforced by construction rather than by external constraint-handling mechanisms.

The mathematical foundations of constrained optimization are rooted in Lagrangian duality, variational analysis, and Karush--Kuhn--Tucker (KKT) theory. Classical references such as \cite{bazaraa2006nonlinear,rockafellar1998variational} provide comprehensive treatments of optimality conditions, constraint qualifications, and duality frameworks. In convex settings, interior-point and barrier methods \cite{wright1992interior,forsgren2002interior} offer algorithms with strong theoretical guarantees and approximate the original constrained problem by a sequence of unconstrained problems. Their numerical realization, however, requires careful tuning of barrier parameters and related algorithmic choices. For nonlinear problems, sequential quadratic programming, augmented Lagrangian methods, and exact penalty formulations \cite{boggs1995sequential,birgin2014practical,di1989exact} are widely used and exhibit superlinear local convergence under suitable regularity assumptions. When these methods are extended to infinite-dimensional settings, additional analytical and numerical challenges arise. In particular, classical numerical approaches often enforce constraints through penalty terms, projection operators, or active-set strategies \cite{malitsky2015projected,hintermuller2002primal}. Penalty-based formulations require careful tuning of weights and may suffer from ill-conditioning as penalty parameters increase \cite{bertsekas2014constrained}.

PDE-constrained optimization has been systematically studied over the past decades; see, for instance, \cite{hinze2008optimization,troltzsch2010optimal,hintermuller2010pde}. Classical solution strategies are typically based on reduced formulations through the control-to-state mapping, adjoint equations for gradient computation, and semismooth Newton or active-set techniques for handling constraints. These approaches form the backbone of efficient gradient-based algorithms in both theory and practice. In practical applications, PDE-constrained optimization problems often involve some control or state constraints. Control constraints, such as box restrictions on distributed controls, are relatively well understood and can often be treated using projection formulas or primal-dual active-set strategies within the KKT framework. In contrast, state constraints introduce substantially greater analytical and numerical difficulties. For elliptic and parabolic problems with pointwise state constraints, the associated Lagrange multipliers may belong to spaces of Radon measures or exhibit low Sobolev regularity rather than lie in $L^p(\Omega)$ \cite{hintermuller2006feasible,troltzsch2010optimal}. Furthermore, classical constraint qualifications may fail in infinite-dimensional settings, and strict complementarity may break down, especially in problems involving complementarity-type structures or biactive sets of positive measure \cite{hintermuller2009mathematical}. These difficulties often lead to instability in multiplier-based algorithms and require sophisticated regularization, interior-point, or active-set strategies.

In recent years, neural network methods have emerged as flexible, mesh-free approaches for solving PDEs and related inverse problems \cite{raissi2019physics,yu2018deep,nganyu2023deep,lu2021learning}. Several neural frameworks have also been developed for PDE-constrained optimal control problems. One class of approaches incorporates the PDE residual as a penalty term in the objective function, as in physics-informed neural networks (PINNs) \cite{lu2021physics,mowlavi2023optimal}. Another class of methods relies on neural approximations of the KKT system. For example, the adjoint-oriented neural network (AONN) proposed in \cite{yin2024aonn} alternately updates the state, adjoint, and control variables. The AOMAD method \cite{yong2026adjoint} combines adjoint techniques with meta-learning strategies to efficiently handle parametric optimal control problems. The OSNN method \cite{dai2025solving} approximates the state and adjoint variables in elliptic optimal control problems using separate networks and solves the resulting reduced optimality system using PINN techniques. In \cite{song2024admm}, the alternating direction method of multipliers was integrated with PINNs to address nonsmooth PDE-constrained optimization problems. Despite their flexibility, enforcing constraints robustly within neural frameworks remains challenging. Existing KKT-based neural approaches mainly address control constraints and become difficult to apply to state-constrained problems when the associated multipliers exhibit low regularity. Moreover, many existing neural approaches enforce the PDE constraint through residual minimization, so that the resulting numerical state and control generally do not satisfy the PDE relation exactly. Another limitation of PINN-based neural approaches for PDE-constrained optimization is that they require sufficient regularity assumptions on the state, which may not be satisfied in many situations, such as irregular domains. We also mention some operator learning methods developed for PDE-constrained optimization where the learned control-to-state mapping is employed to search for optimal control; see \cite{hwang2022solving,dong2022optimization}.

Beyond PDE solvers, considerable effort has been devoted to designing neural network architectures that preserve structural properties of the target functions. Examples include divergence-free constructions for incompressible flows via stream functions \cite{raissi2019physics}, hard enforcement of Dirichlet and periodic boundary conditions \cite{lu2021physics}, invertible neural networks \cite{behrmann2019invertible}, and input-convex neural networks for convexity preservation \cite{amos2017input}. These developments demonstrate that embedding constraints directly into the architecture, rather than enforcing them through penalties, can significantly improve numerical robustness and interpretability. While existing structure-preserving architectures typically enforce specific functional properties, much less attention has been devoted to the geometric structure of admissible sets arising in constrained optimization. A particularly important example is the class of polyhedral constraint sets defined by finitely many linear inequalities. Such sets play a fundamental role in convex optimization. However, constructive parameterizations of polyhedral sets suitable for numerical computation in infinite-dimensional settings remain largely unexplored. 

The preceding discussion reveals three limitations of current neural approaches to constrained optimization. First, penalty-based methods enforce feasibility only approximately and require delicate tuning of loss weights. Second, KKT-based neural solvers introduce additional networks for adjoints and multipliers, and their training may become unstable in the presence of low-regularity multipliers or degenerate complementarity conditions. Third, to the best of our knowledge, existing approaches do not provide a unified geometric framework for encoding general convex polyhedral constraints, nor a hard-constraint formulation for exact satisfaction of PDE constraints.

The main contributions of this paper are summarized as follows:
\begin{enumerate}
\item We develop a Hilbert-space decomposition framework for polyhedral constraint sets and derive constructive smooth CNP schemes that enforce feasibility exactly, including extensions to Dirichlet boundary conditions. We also propose CNPs for several classes of pointwise constraints, with and without boundary conditions.

\item For PDE-constrained optimization problems whose PDE constraints admit an explicit reduced form, we introduce an exact reduced CNP that enforces the PDE constraint by construction, reduces the number of neural networks, and avoids multiplier-related difficulties. For Poisson equations in polygonal domains, we further propose a singularity-enriched CNP to capture low-regularity optimal states.

\item We validate the proposed methods on state-constrained, low-regularity, complementarity-type PDE control problems and inverse problems, demonstrating improved accuracy and feasibility preservation.
\end{enumerate}

The remainder of this paper is organized as follows. 
Section~\ref{sec:polyhedral} presents the theoretical foundations for polyhedra in Hilbert spaces and develops smooth CNP constructions, including extensions to indirectly defined constraint functionals and Dirichlet boundary conditions. 
Section~\ref{sec:pointwise} introduces CNP schemes for several pointwise constraints. 
Section~\ref{sec:PDE} focuses on PDE-constrained optimization, presenting the exact reduced neural method and its singularity-enriched modification. 
Section~\ref{sec:numerical} reports numerical experiments on several challenging examples. 
Finally, Section~\ref{sec:conclusion} concludes the paper and discusses directions for future research.

\section{Polyhedral constraints}
\label{sec:polyhedral}

In this section, we briefly recall the definition of polyhedra in finite-dimensional spaces and the classical Minkowski--Weyl theorem. We then extend the discussion to infinite-dimensional settings and present some decomposition formulas for polyhedra in Hilbert spaces. This result serves as the theoretical foundation for the constrained neural parameterization of polyhedral sets developed in this section.

\subsection{Decomposition of polyhedral sets}

Let $d\in \mathbb{N}$ be a positive integer. A polyhedron in $\mathbb{R}^d$ is a convex set defined as the intersection of finitely many half-spaces \cite{boyd2004convex}:
\begin{equation}
    \mathcal{P}=\{x\in \mathbb{R}^d \mid a_i^\top x\le b_i,\ i=1,\dots,n\}.
\end{equation}
The convex hull of a set $U\subset \mathbb{R}^d$ is defined by \cite{boyd2004convex}
\begin{equation}
    \operatorname{conv}(U)
    :=
    \left\{
    \sum_{i=1}^m \lambda_i x_i
    \,\middle|\,
    x_i\in U,\ \lambda_i\ge 0,\ i=1,\dots,m,\ \sum_{i=1}^m \lambda_i=1
    \right\},
\end{equation}
that is, the set of all finite convex combinations of points in $U$. The conic hull of $U$ is defined as \cite{boyd2004convex}
\begin{equation}
    \operatorname{cone}(U)
    :=
    \left\{
    \sum_{i=1}^m \lambda_i x_i
    \,\middle|\,
    x_i\in U,\ \lambda_i\ge 0,\ i=1,\dots,m
    \right\}.
\end{equation}

We next recall the classical Minkowski--Weyl theorem \cite{boyd2004convex}.

\begin{theorem}[Minkowski--Weyl theorem]\label{theorem:MW_theorem}
Let $\mathcal{P}\subseteq \mathbb{R}^d$ be a polyhedron. Then there exist finitely many points $\{y_i\}_{i=1}^p\subset \mathbb{R}^d$ and finitely many vectors $\{z_i\}_{i=1}^q\subset \mathbb{R}^d$ such that
\begin{equation}
    \mathcal{P}
    =
    \operatorname{conv}(\{y_i\}_{i=1}^p)
    +
    \operatorname{cone}(\{z_i\}_{i=1}^q).
    \label{equa:MW_theorem}
\end{equation}
Conversely, any set of the form \eqref{equa:MW_theorem} is a polyhedron.
\end{theorem}

Let now $X$ be a Banach space with dual space $X^*$. A subset $\mathcal{P}\subset X$ is called a polyhedron if there exist $\alpha_1,\dots,\alpha_n\in X^*$ and $b_1,\dots,b_n\in \mathbb{R}$ such that
\begin{equation}
    \mathcal{P}
    :=
    \{x\in X \mid \langle \alpha_i,x\rangle_{X^*,X}\le b_i,\ i=1,\dots,n\}.
    \label{equa:polyhedron_inf}
\end{equation}
The following decomposition theorem for polyhedra in Banach spaces was established in \cite{zheng2009pareto,luan2020representation}.

\begin{theorem}
Let $X$ be a Banach space and let $\mathcal{P}$ be a polyhedron of the form \eqref{equa:polyhedron_inf}. Then there exist closed subspaces $X_1,X_2\subset X$ and a polyhedron $\mathcal{P}_2\subset X_2$ such that
\begin{equation}
    \mathcal{P}=X_1+\mathcal{P}_2,
\end{equation}
where
\begin{equation}
    X_1=\{x\in X \mid \langle \alpha_i,x\rangle_{X^*,X}=0,\ i=1,\dots,n\},
\end{equation}
\begin{equation}
    X=X_1+X_2,\qquad X_1\cap X_2=\{0\},\qquad \dim(X_2)<\infty.
    \label{equa:x1x2}
\end{equation}
Conversely, if $X_1,X_2$ satisfy \eqref{equa:x1x2} and $\mathcal{P}_2\subset X_2$ is a polyhedron, then $\mathcal{P}=X_1+\mathcal{P}_2$ is a polyhedron in $X$.
\end{theorem}

We now specialize to the Hilbert space setting. Using the Riesz representation theorem, we identify $X^*$ with $X$. In this case, the structure of polyhedra admits a more explicit representation.

\begin{theorem}\label{Theo:decomp_hilbert}
Let $X$ be a Hilbert space and let $\mathcal{P}\subset X$ be the polyhedron
\begin{equation}\label{equa:polyhedron_hilb}
    \mathcal{P}
    :=
    \{x\in X \mid \langle \alpha_i,x\rangle_X\le b_i,\ i=1,\dots,n\},
    \qquad \alpha_1,\dots,\alpha_n\in X.
\end{equation}
Denote $V_\alpha=\operatorname{Span}\{\alpha_1,...,\alpha_n\}$ and $V_\alpha^\bot$ be its orthogonal complement in $X$. Then there exist finitely many elements $\{y_i\}_{i=1}^p\subset V_\alpha$ and  $\{z_i\}_{i=1}^q\subset V_\alpha$ such that
\begin{equation}\label{equa:MW_theorem_inf}
    \mathcal{P}
    =
    V_\alpha^\perp
    +
    \operatorname{conv}(\{y_i\}_{i=1}^p)
    +
    \operatorname{cone}(\{z_i\}_{i=1}^q).
\end{equation}
Conversely, let $V\subset X$ be a finite-dimensional subspace and suppose that
\begin{equation}\label{mh:1}
\mathcal{P}
=
V^\perp
+
\operatorname{conv}(\{y_i\}_{i=1}^p)
+
\operatorname{cone}(\{z_i\}_{i=1}^q),
\qquad
\{y_i\}_{i=1}^p,\{z_i\}_{i=1}^q\subset V.
\end{equation}
Then $\mathcal{P}$ is a polyhedron of the form \eqref{equa:polyhedron_hilb} for some $\{\alpha_i\}_{i=1}^n\subset V$ and $\{b_i\}_{i=1}^n\subset \mathbb{R}$.
\end{theorem}

\begin{proof}
Firstly, for a polyhedron $\mathcal{P}$ given as \eqref{equa:polyhedron_hilb}, we note that $\mathcal{P}$ has the following decomposition
\begin{equation}
            \mathcal{P}=V_\alpha^\bot+ \{x\in V_\alpha|\langle \alpha_i,x\rangle_X\le b_i,i=1,...,n\}.
     \label{equa:proof1}
\end{equation}
Define
$
\mathcal{P}_\alpha
:=
\{x\in V_\alpha \mid \langle \alpha_i,x\rangle_X\le b_i,\ i=1,\dots,n\}$, which is a pointed polyhedron in the finite-dimensional space $V_\alpha$. By Theorem~\ref{theorem:MW_theorem}, there exist finitely many elements $\{y_i\}_{i=1}^p\subset V_\alpha$ and $\{z_i\}_{i=1}^q\subset V_\alpha$ such that
\begin{equation}
    \mathcal{P}_\alpha
    =
    \operatorname{conv}(\{y_i\}_{i=1}^p)
    +
    \operatorname{cone}(\{z_i\}_{i=1}^q).
    \label{equa:proof2}
\end{equation}
Combining \eqref{equa:proof1} and \eqref{equa:proof2} yields \eqref{equa:MW_theorem_inf}.

For the converse direction, let $V\subset X$ be finite-dimensional and suppose that \eqref{mh:1} holds true.
By Theorem~\ref{theorem:MW_theorem}, the set
$
\operatorname{conv}(\{y_i\}_{i=1}^p)
+
\operatorname{cone}(\{z_i\}_{i=1}^q)
$
is a polyhedron in $V$. Hence, there exist  $\{\alpha_i\}_{i=1}^n\subset V$ and $\{b_i\}_{i=1}^n\subset \mathbb{R}$ such that
%\[
$\{x\in V \mid \langle \alpha_i,x\rangle_X\le b_i,\ i=1,\dots,n\}
=
\operatorname{conv}(\{y_i\}_{i=1}^p)
+
\operatorname{cone}(\{z_i\}_{i=1}^q)$.
%\]
Consequently,
\begin{equation}
\mathcal{P}
=
V^\perp
+
\{x\in V \mid \langle \alpha_i,x\rangle_X\le b_i,\ i=1,\dots,n\}=
\{x\in X \mid \langle \alpha_i,x\rangle_X\le b_i,\ i=1,\dots,n\}.
\end{equation}
This proves the claim.
\end{proof}

In the decomposition \eqref{equa:MW_theorem_inf}, all lineality directions are absorbed into the orthogonal complement $V_\alpha^\perp$. Consequently, the remaining finite-dimensional polyhedron $\mathcal{P}_\alpha
=
\{x\in V_\alpha \mid \langle \alpha_i,x\rangle_X\le b_i,\ i=1,\dots,n\}$ is pointed, i.e., it contains no affine line. We record this observation in the following lemma.

\begin{lemma}\label{lem:pointed}
Let $V_\alpha=\operatorname{Span}\{\alpha_1,\dots,\alpha_n\}$. Then the set
%\begin{equation}
    $\mathcal{P}_\alpha=
    \{x\in V_\alpha \mid \langle \alpha_i,x\rangle_X\le b_i,\ i=1,\dots,n\}$
%\end{equation}
is a pointed polyhedron in $V_\alpha$.
\end{lemma}

\begin{proof}
The lineality space of $\mathcal{P}_\alpha$ is given by
%\[
$\operatorname{lin}(\mathcal{P}_\alpha)
:=
\{x\in V_\alpha \mid \langle \alpha_i,x\rangle_X=0,\ i=1,\dots,n\}$.
%\]
If $x\in \operatorname{lin}(\mathcal{P}_\alpha)$, then $x\in V_\alpha=\operatorname{Span}\{\alpha_1,\dots,\alpha_n\}$ and $x$ is orthogonal to every $\alpha_i$. Hence $x$ is orthogonal to $V_\alpha$, and therefore $x=0$. Thus, $\operatorname{lin}(\mathcal{P}_\alpha)=\{0\}.$ Since a polyhedron is pointed if and only if it contains no affine line, equivalently if and only if its lineality space is trivial, the result follows.
\end{proof}

%\subsubsection*{Computation of vertices and extreme rays}

{\bf Computation of vertices and extreme rays.} Theorem~\ref{Theo:decomp_hilbert} yields a decomposition of $\mathcal{P}$ in which the sets 
$\{y_i\}_{i=1}^p$ and $\{z_i\}_{i=1}^q$ correspond to the vertices and extreme rays of the pointed polyhedron $\mathcal{P}_\alpha$, respectively. Let $r:=\dim(V_\alpha)$. Following \cite{bertsekas2003convex}, vertices of $\mathcal{P}_\alpha$ can be obtained by examining active constraint sets of size $r$. 
Specifically, for each index set
%\[
$J=\{j_1,\dots,j_r\}\subseteq \{1,\dots,n\}$,
%\]
such that the vectors $\{\alpha_{j_k}\}_{k=1}^r$ are linearly independent, we solve the following system in $V_\alpha$:
\[
\langle \alpha_{j_k},y\rangle_X=b_{j_k},\qquad k=1,\dots,r.
\]
Then $y$ is a vertex of $\mathcal{P}_\alpha$ if the resulting point satisfies  $y\in V_\alpha\cap\mathcal{P}$.

Similarly, we obtain the extreme rays by examining active constraint sets of size $r-1$. 
For each index set
%\[
$J=\{j_1,\dots,j_{r-1}\}\subseteq\{1,\dots,n\}$,
%\]
with linearly independent $\{\alpha_{j_k}\}_{k=1}^{r-1}$, we compute a nonzero vector $z\in V_\alpha$ spanning the one-dimensional space
%\[
$\{x\in V_\alpha\mid \langle \alpha_{j_k},x\rangle_X=0,\ k=1,\dots,r-1\}$.
%\]
If $z$ satisfies
%\[
$\langle \alpha_i,z\rangle_X\le 0$, $i=1,\dots,n$,
%\]
then $z$ generates an extreme ray of $\mathcal{P}_\alpha$. 
If instead $-z$ satisfies the above inequalities, then $-z$ generates an extreme ray of $\mathcal{P}_\alpha$.

\subsection{Constrained neural parameterization of polyhedra}

Let $X$ be a Hilbert space of functions defined on a bounded domain $\Omega\subset \mathbb{R}^d$, and let $\mathcal{P}$ be a polyhedron represented as in \eqref{equa:MW_theorem_inf}. We now introduce a constrained neural parameterization for $\mathcal{P}$. To this end, we introduce two families of scalar functions, $\{g_i\}_{i=1}^p$ and $\{h_i\}_{i=1}^q$, satisfying the following conditions.

\begin{condition}[Conditions on $\{g_i\}_{i=1}^p$ and $\{h_i\}_{i=1}^q$]
\label{conditions}
\leavevmode\par
\begin{enumerate}
    \item For each $i$,  $g_i:\mathbb{R}^p\to\mathbb{R}$ and $h_i:\mathbb{R}\to\mathbb{R}$ are differentiable.
    \item The image of the mapping $(g_1,\dots,g_p):\mathbb{R}^p\to\mathbb{R}^p$ is dense in the probability simplex
    \begin{equation}
        \mathcal{C}:=\left\{(c_1,\dots,c_p)\in\mathbb{R}^p \,\middle|\, c_i\ge 0,\ i=1,\dots,p,\ \sum_{i=1}^p c_i=1\right\}.
    \end{equation}
    \item For each $i=1,\dots,q$, the range of $h_i$ is a dense subset of $[0,\infty)$.
\end{enumerate}
\end{condition}

Suitable choices for $\{g_i\}_{i=1}^p$ include the softmax mapping
\begin{equation}
    g_i(\lambda_1,\dots,\lambda_p)
    =
    \frac{e^{\lambda_i}}{\sum_{j=1}^p e^{\lambda_j}},
    \qquad i=1,\dots,p,
\end{equation}
or the recursive construction
\begin{equation}
    g_1 = \sin^2(\lambda_1),\, g_2=(1-g_1)\sin^2(\lambda_2),...,\,g_{p-1}=\big(1-\sum_{i=1}^{p-2}g_i\big)\sin^2(\lambda_{p-1}),\,g_p = 1-\sum_{i=1}^{p-1}g_i.
\end{equation}
The second construction maps onto the entire simplex $\mathcal{C}$, whereas the softmax mapping yields strictly positive coefficients and therefore does not attain the boundary of $\mathcal{C}$. Representative choices for $\{h_i\}_{i=1}^q$ include
\begin{equation}
    h_i(\gamma_i)=\gamma_i^2,
    \qquad\text{or}\qquad
    h_i(\gamma_i)=e^{\gamma_i},
    \qquad i=1,\dots,q.
\end{equation}
To construct a constrained neural parameterization of $\mathcal{P}$, we introduce a neural network $\mathcal{N}(\cdot,\theta)$ with trainable parameter set $\theta$ and a suitable activation function $\sigma$ such that $\mathcal{N}(\cdot,\theta)\in X$, together with trainable scalars $\{\lambda_i\}_{i=1}^p$ and $\{\gamma_i\}_{i=1}^q$; see e.g., \cite{lecun2015deep} for a definition of neural networks.  For simplicity, we suppress the explicit dependence on the activation function $\sigma$ in the notation $\mathcal{N}(\cdot,\theta)$. Let $\{\hat{\alpha}_i\}_{i=1}^m$ be an orthonormal basis of $V_\alpha:=\operatorname{Span}\{\alpha_1,\dots,\alpha_n\}$. We define
\begin{equation}\label{equa:CNP_polyhedra}
     U(\cdot,\hat{\theta})=\mathcal{N}(\cdot,\theta)-\sum_{i=1}^m \langle \mathcal{N}(\cdot,\theta),\hat{\alpha}_i\rangle_X\hat{\alpha}_i+ \sum_{i=1}^p g_i(\lambda_1,...,\lambda_p) y_i + \sum_{i=1}^q h_i(\gamma_i) z_i,\quad \hat{\theta} = \{\theta, \{\lambda_i\},\{\gamma_i\}\}.
\end{equation}
This construction is directly motivated by the decomposition in Theorem~\ref{Theo:decomp_hilbert}: the first two terms project the neural network onto the orthogonal complement $V_\alpha^\perp$, while the remaining terms parameterize the pointed polyhedron $\mathcal{P}_\alpha$. Moreover, the gradient of $U(x,\hat{\theta})$ is given by
\begin{equation}
   \nabla_x U(x,\hat{\theta})=
   \nabla_x \mathcal{N}(x,\theta)
   -
   \sum_{i=1}^m \langle \mathcal{N}(\cdot,\theta),\hat{\alpha}_i\rangle_X \nabla_x \hat{\alpha}_i(x)
   +
   \sum_{i=1}^p g_i(\lambda_1,\dots,\lambda_p)\nabla_x y_i(x)
   +
   \sum_{i=1}^q h_i(\gamma_i)\nabla_x z_i(x),
\end{equation}

We assume that the architecture $\mathcal{N}(\cdot,\theta)$ possesses the universal approximation property in $X$ with respect to the norm $\|\cdot\|_X$, that is, the family $\{\mathcal{N}(\cdot,\theta):\theta\}$ is dense in $X$. Such density results are well established for many classes of feedforward neural networks with non-polynomial activation functions in common function spaces; see, for example, \cite{hornik1991approximation,hornik1993some,zhang2024deep}. Under this assumption, the constrained neural parameterization \eqref{equa:CNP_polyhedra} inherits a universal approximation property on $\mathcal{P}$.

\begin{lemma}
Let $U(\cdot,\hat{\theta})$ be defined by \eqref{equa:CNP_polyhedra}, and suppose that the functions $\{g_i\}_{i=1}^p$ and $\{h_i\}_{i=1}^q$ satisfy Condition~\ref{conditions}. Then $U(\cdot,\hat{\theta})\in\mathcal{P}$ for every $\hat{\theta}$. Moreover, if $\mathcal{N}(\cdot,\theta)$ is a universal approximator in $X$, then the family $\{U(\cdot,\hat{\theta}):\hat{\theta}\}$ is dense in $\mathcal{P}$ with respect to the norm $\|\cdot\|_X$.
\end{lemma}

\begin{proof}
Let $U(\cdot,\hat{\theta})$ be given by \eqref{equa:CNP_polyhedra}. It is easy to verify that
\begin{equation}
\begin{split}
    \begin{cases}
       \langle \mathcal{N}(\cdot,\theta)-\sum_{i=1}^m \langle \mathcal{N}(\cdot,\theta),\hat{\alpha}_i\rangle_X \hat{\alpha}_i,\alpha_j\rangle_X=0,\quad &j=1,...,n,\\
\langle \sum_{i=1}^p g_i(\lambda_1,...,\lambda_p) y_i,\alpha_j\rangle_X\le b_j.\quad &j=1,...,n,       \\
\langle\sum_{i=1}^q h_i(\gamma_i) z_i,\alpha_j\rangle_X\le 0,\quad & j=1,...,n.
    \end{cases}
    \end{split}
\end{equation}
Therefore $U(\cdot,\hat{\theta})\in \mathcal{P}$. Next, let $u\in\mathcal{P}$ be arbitrary. By Theorem~\ref{Theo:decomp_hilbert}, there exist
\[
u^\perp\in V_\alpha^\perp,\qquad (g_1^*,\dots,g_p^*)\in\mathcal{C},\qquad h_i^*\ge 0,\ i=1,\dots,q,
\]
such that $u=u^\perp+\sum_{i=1}^p g_i^* y_i+\sum_{i=1}^q h_i^* z_i$. Since the image of $(g_1,\dots,g_p)$ is dense in $\mathcal{C}$, the range of each $h_i$ is dense in $[0,\infty)$, and the family $\{\mathcal{N}(\cdot,\theta):\theta\}$ is dense in $X$, it follows that $\{U(\cdot,\hat{\theta}):\hat{\theta}\}$ is dense in $\mathcal{P}$.
\end{proof}

\begin{remark}The inner products in the CNP \eqref{equa:CNP_polyhedra} are evaluations of the original constraint functionals, which are also required in methods such as penalty, augmented Lagrangian, and primal-dual methods. Thus the CNP does not introduce an additional type of constraint evaluation. The vertices and extreme rays are computed only once for a finite-dimensional polyhedron, and this offline step is inexpensive when the number of constraints is moderate.
\end{remark}

\subsubsection{Handling indirectly specified constraint functionals}\label{sec:alpha_i}

In certain applications, the elements $\{\alpha_i\}\subset X$ defining the polyhedron \eqref{equa:polyhedron_hilb} may not be directly available. For example, let $X=H^1(\Omega)$, the usual Sobolev space \cite{FournierAdams} with some domain $\Omega\subseteq\mathbb{R}^d$, and let $\Omega_1\subset\Omega$. Consider the constraint set
\begin{equation}\label{equa:cons_indi}
    \{u\in H^1(\Omega)\mid \langle \chi_{\Omega_1},u\rangle_{L^2(\Omega)}\le b\},
\end{equation}
where $\chi_{\Omega_1}:\Omega\to\mathbb{R}$ denotes the characteristic function of $\Omega_1$; i.e., $\chi_{\Omega_1}(x)=1$ if $x\in\Omega_1$ and $\chi_{\Omega_1}(x)=0$ otherwise.
Since we may have $\chi_{\Omega_1}\notin H^1(\Omega)$, the parameterization introduced above cannot be applied directly. Nevertheless, by the Riesz representation theorem, there exists a unique $\alpha\in H^1(\Omega)$ such that
\begin{equation}\label{equa:riesz}
    \langle \alpha,v\rangle_{H^1(\Omega)}
    =
    \langle \chi_{\Omega_1},v\rangle_{L^2(\Omega)},
    \qquad \forall v\in H^1(\Omega),
\end{equation}
and the constraint set \eqref{equa:cons_indi} is therefore equivalent to
%\begin{equation}
    $\{u\in H^1(\Omega)\mid \langle \alpha,u\rangle_{H^1(\Omega)}\le b\}$.
%\end{equation}
Since solving \eqref{equa:riesz} for $\alpha$  may be computationally expensive in high-dimensional settings or large-scale problems, we introduce a more efficient construction for a class of such constraints. For the sake of generality, let $X$ and $Y$ be Hilbert spaces such that $X\subset Y$. Consider the constraint set
\begin{equation}\label{equa:polydra_general}
    \mathcal{P}
    =
    \{u\in X\mid \langle \alpha_i,u\rangle_Y\le b_i,\ \alpha_i\in Y,\ i=1,\dots,n\}.
\end{equation}
Before presenting the parameterization, we establish the following auxiliary result.

\begin{lemma}\label{lemma:orthocond}
Let $X$ and $Y$ be Hilbert spaces with $X\subset Y$. Let $\{\hat{\alpha}_i\}_{i=1}^m$ be an orthonormal family in $Y$ satisfying
\begin{equation}
\sum_{i=1}^m \operatorname{dist}_Y(\hat{\alpha}_i,X)<1,
\qquad\text{where}\qquad 
\operatorname{dist}_Y(\alpha,X):=\inf_{x\in X}\|\alpha-x\|_Y,\quad \alpha\in Y.
\end{equation}
Then there exist elements $\{\beta_i\}_{i=1}^m\subset X$ such that
\begin{equation}\label{equa:orthocod2}
\langle \hat{\alpha}_i,\beta_j\rangle_Y=\delta_{ij},\qquad i,j=1,\dots,m,
\end{equation}
with $\delta_{ij}$ denoting the Kronecker delta; i.e., $\delta_{ij}=1$ if $i=j$ and $\delta_{ij}=0$ else.
\end{lemma}

\begin{proof}
By the distance condition, there exist $\{\gamma_i\}_{i=1}^m\subset X$ satisfying $\sum_{i=1}^m\|\gamma_i-\hat{\alpha}_i\|_Y<1.$ Therefore
\begin{equation}
|\langle \hat{\alpha}_i,\gamma_i\rangle_Y|
=|\langle \hat{\alpha}_i,\gamma_i-\hat{\alpha}_i+\hat{\alpha}_i\rangle_Y|
=|1+\langle \hat{\alpha}_i,\gamma_i-\hat{\alpha}_i\rangle_Y|
\ge 1-\|\gamma_i-\hat{\alpha}_i\|_Y,\quad i=1,...,m.
\end{equation}
For $i\neq j$, we have
\begin{equation}
|\langle \hat{\alpha}_i,\gamma_j\rangle_Y|
=|\langle \hat{\alpha}_i,\gamma_j-\hat{\alpha}_j+\hat{\alpha}_j\rangle_Y|
=|\langle \hat{\alpha}_i,\gamma_j-\hat{\alpha}_j\rangle_Y|
\le \|\gamma_j-\hat{\alpha}_j\|_Y.
\end{equation}
Consequently,
\begin{equation}
|\langle \hat{\alpha}_i,\gamma_i\rangle_Y|
> \sum_{j\neq i}|\langle \hat{\alpha}_i,\gamma_j\rangle_Y|,\qquad i=1,\dots,m.
\end{equation}
Thus the matrix $A=(A_{ij})_{i,j=1}^m$ defined by $A_{ij}:=\langle \hat{\alpha}_i,\gamma_j\rangle_Y$ is strictly diagonally dominant and therefore invertible. Let $B=A^{-1}$ and define
\begin{equation}
\beta_j:=\sum_{k=1}^m B_{kj}\gamma_k\in X,\qquad j=1,\dots,m.
\end{equation}
Then
\begin{equation}
\langle \hat{\alpha}_i,\beta_j\rangle_Y
=\sum_{k=1}^m B_{kj}\langle \hat{\alpha}_i,\gamma_k\rangle_Y
=\sum_{k=1}^m A_{ik}B_{kj}
=\delta_{ij},
\end{equation}
which completes the proof.
\end{proof}

\begin{remark}
The hypothesis of Lemma~\ref{lemma:orthocond} is satisfied in many situations. For instance, when $X=H^k(\Omega)$ and $Y=H^l(\Omega)$ with $k>l$, the space $X$ is dense in $Y$, which is a sufficient condition for the hypothesis. To construct the elements $\{\gamma_i\}$ in the proof, observe that if $\hat{\alpha}_i\in X$, we may simply set $\gamma_i=\hat{\alpha}_i$. Otherwise, we can choose a finite-dimensional subspace $X_f\subset X$ and let $\gamma_i$ be the orthogonal projection of $\hat{\alpha}_i$ onto $X_f$ with respect to the inner product of $Y$. By selecting $X_f$ appropriately, the distance condition can be ensured.
\end{remark}

\begin{remark}\label{remark2}
    When $X$ is not dense in $Y$, the hypothesis of Lemma~\ref{lemma:orthocond} may not hold. This issue can be handled by observing that the constraint set \eqref{equa:polydra_general} is equivalent to 
%\[
 $\mathcal{P}=\{u\in X\mid \langle \tilde{\alpha}_i,u\rangle_Y \le b_i,\tilde{\alpha}_i\in \overline{X}^Y, i=1,...,n\}$,
%\]
where $\overline{X}^Y$ is the closure of $X$ in $Y$, and $\tilde{\alpha}_i$ is the orthogonal projection of $\alpha_i$ onto the space $\overline{X}^Y$ in $Y$. Note that $\overline{X}^Y$ is also a Hilbert space with the inner product $\langle\cdot,\cdot \rangle_{Y}$. Then by replacing $Y$ with $\overline{X}^Y$, the hypothesis of Lemma~\ref{lemma:orthocond} is automatically satisfied.
\end{remark}

Let $\{\hat{\alpha}_i\}_{i=1}^m$ be an orthonormal basis of $\operatorname{Span}\{\alpha_1,\dots,\alpha_n\}$ in $Y$. By Remark~\ref{remark2}, without loss of generality, we assume the hypothesis of Lemma~\ref{lemma:orthocond} is satisfied. By Lemma~\ref{lemma:orthocond}, there exist $\{\beta_i\}_{i=1}^m\subset X$ such that $ \langle \hat{\alpha}_i,\beta_j\rangle_Y=\delta_{ij},\, i,j=1,\dots,m.$ Define
\begin{equation}\label{equa:V1}
    V_1
    :=
    \left\{
    u-\sum_{i=1}^m \langle \hat{\alpha}_i,u\rangle_Y \beta_i
    \,\middle|\,
    u\in X
    \right\},
\end{equation}
\begin{equation}\label{equa:V2}
    V_2
    :=
    \{u\in X_\beta\mid \langle \alpha_i,u\rangle_Y\le b_i,\ i=1,\dots,n\},
    \qquad
    X_\beta:=\operatorname{Span}\{\beta_1,\dots,\beta_m\}.
\end{equation}
The next lemma provides a decomposition of the constraint set \eqref{equa:polydra_general}.

\begin{lemma}\label{lemma:decom_general}
Let $\mathcal{P}$, $V_1$, and $V_2$ be defined by \eqref{equa:polydra_general}, \eqref{equa:V1}, and \eqref{equa:V2}, respectively. Then
\[
\mathcal{P}=V_1+V_2.
\]
Moreover, there exist finitely many points $\{y_i\}_{i=1}^p\subset X_\beta$ and finitely many vectors $\{z_i\}_{i=1}^q\subset X_\beta$ such that
\begin{equation}\label{equa:decomll}
    V_2
    =
    \operatorname{conv}(\{y_i\}_{i=1}^p)
    +
    \operatorname{cone}(\{z_i\}_{i=1}^q).
\end{equation}
\end{lemma}

\begin{proof}
We first prove that $V_1+V_2\subseteq\mathcal{P}$. Let $v_1\in V_1$. By definition, there exists $u\in X$ such that $
v_1=u-\sum_{i=1}^m \langle \hat{\alpha}_i,u\rangle_Y \beta_i.$
By the biorthogonality relation \eqref{equa:orthocod2}, for each $j=1,\dots,m$, we have
\begin{equation*}
    \langle \hat{\alpha}_j,v_1\rangle_Y=
\left\langle
\hat{\alpha}_j,
u-\sum_{i=1}^m \langle \hat{\alpha}_i,u\rangle_Y \beta_i
\right\rangle_Y=
\langle \hat{\alpha}_j,u\rangle_Y
-
\sum_{i=1}^m \langle \hat{\alpha}_i,u\rangle_Y \langle \hat{\alpha}_j,\beta_i\rangle_Y =
\langle \hat{\alpha}_j,u\rangle_Y-\langle \hat{\alpha}_j,u\rangle_Y
=0.
\end{equation*}
Since $\{\hat{\alpha}_i\}_{i=1}^m$ forms an orthonormal basis of $\operatorname{Span}\{\alpha_1,\dots,\alpha_n\}$ in $Y$, it follows that
$\langle \alpha_j,v_1\rangle_Y=0$ for $j=1,\dots,n$. Therefore, for any $v_2\in V_2$, we obtain
\begin{equation}
\langle \alpha_j,v_1+v_2\rangle_Y
=
\langle \alpha_j,v_1\rangle_Y+\langle \alpha_j,v_2\rangle_Y
\le b_j,
\qquad j=1,\dots,n,    
\end{equation}
which implies that $v_1+v_2\in\mathcal{P}$. Hence $V_1+V_2\subseteq\mathcal{P}$.

Conversely, let $u\in\mathcal{P}$. Define
\[
u_1:=u-\sum_{i=1}^m \langle \hat{\alpha}_i,u\rangle_Y\beta_i
\in V_1, \quad u_2:=\sum_{i=1}^m \langle \hat{\alpha}_i,u\rangle_Y\beta_i
\in X_\beta.
\]
It remains to show that $u_2\in V_2$. Since each $\alpha_j\in \operatorname{Span}\{\hat{\alpha}_1,\dots,\hat{\alpha}_m\}$, we may write
$\alpha_j=\sum_{k=1}^m \langle \alpha_j,\hat{\alpha}_k\rangle_Y \hat{\alpha}_k.$ Using \eqref{equa:orthocod2}, we obtain
\begin{equation*}
\begin{aligned}
\langle \alpha_j,u_2\rangle_Y
&=
\left\langle
\sum_{k=1}^m \langle \alpha_j,\hat{\alpha}_k\rangle_Y \hat{\alpha}_k,
\sum_{i=1}^m \langle \hat{\alpha}_i,u\rangle_Y \beta_i
\right\rangle_Y \\
&=
\sum_{i=1}^m \sum_{k=1}^m
\langle \alpha_j,\hat{\alpha}_k\rangle_Y
\langle \hat{\alpha}_i,u\rangle_Y
\langle \hat{\alpha}_k,\beta_i\rangle_Y \\
&=
\sum_{k=1}^m \langle \alpha_j,\hat{\alpha}_k\rangle_Y \langle \hat{\alpha}_k,u\rangle_Y
=
\langle \alpha_j,u\rangle_Y
\le b_j,
\qquad j=1,\dots,n.
\end{aligned}
\end{equation*}
Thus $u_2\in V_2$, and therefore $u=u_1+u_2\in V_1+V_2$. This proves $\mathcal{P}=V_1+V_2$.

Finally, since $X_\beta$ is finite-dimensional and $V_2$ is a polyhedron in $X_\beta$, the representation \eqref{equa:decomll} follows from the Minkowski--Weyl theorem.
\end{proof}

With Lemma~\ref{lemma:decom_general} in hand, we can construct a CNP for the polyhedron \eqref{equa:polydra_general}. Let $\mathcal{N}(\cdot,\theta)$ be a neural network with trainable parameter set $\theta$, and let $\{\lambda_i\}_{i=1}^p$ and $\{\gamma_i\}_{i=1}^q$ be trainable scalars. We define
\begin{equation}
    U(\cdot,\hat{\theta})
    =
    \mathcal{N}(\cdot,\theta)
    -
    \sum_{i=1}^m \langle \mathcal{N}(\cdot,\theta),\hat{\alpha}_i\rangle_Y \beta_i
    +
    \sum_{i=1}^p g_i(\lambda_1,\dots,\lambda_p)\,y_i
    +
    \sum_{i=1}^q h_i(\gamma_i)\,z_i, \quad \hat{\theta} = \{\theta, \{\lambda_i\},\{\gamma_i\}\}.
\end{equation}

\subsubsection{Polyhedra with Dirichlet boundary conditions}\label{sec:poly_boundary}

Let $X\subset Y$ be two Hilbert spaces. We consider polyhedra coupled with a Dirichlet boundary condition:
\begin{equation}\label{equa:bound_cons}
    \mathcal{P}^b
    =
    \{u\in X\mid u=g \text{ on }\partial\Omega,\ \langle \alpha_i,u\rangle_Y\le b_i,\ \alpha_i\in Y,\ i=1,\dots,n\}.
\end{equation}
Given a function $\hat{g}\in X$ such that $\hat{g}=g$ on $\partial\Omega$, the set $\mathcal{P}^b$ can be rewritten as
\begin{equation}
    \mathcal{P}^b
    =
    \{u+\hat{g}\in X\mid u=0 \text{ on }\partial\Omega,\ \langle \alpha_i,u\rangle_Y\le c_i,\ \alpha_i\in Y,\ i=1,\dots,n\},
\end{equation}
where $c_i:=b_i-\langle \alpha_i,\hat{g}\rangle_Y,$ for $ i=1,\dots,n$.
We consider $X=H^k(\Omega)$ and let $X_0:=H_0^k(\Omega)$; see, e.g., \cite{FournierAdams} for a definition of these Sobolev spaces. It is therefore sufficient to consider the constraint set
\begin{equation}\label{equa:poly_boundary}
    \mathcal{P}^0
    =
    \{u\in X_0\mid \langle \alpha_i,u\rangle_Y\le b_i,\ \alpha_i\in Y,\ i=1,\dots,n\}.
\end{equation}
Since $X_0$ is a Hilbert space under the inherited inner product from $X$, the problem reduces to the general form \eqref{equa:polydra_general}, with $X$ replaced by $X_0$.

Following the construction in \cite{lu2021physics}, we introduce a smooth boundary-lifting function $w$ satisfying
\begin{equation}\label{equa:boundry}
    w(x)=0 \quad \text{on } \partial\Omega,
    \qquad
    w(x)>0 \quad \text{in } \Omega.
\end{equation}
Then any function of the form $w\phi$, with $\phi\in C^\infty(\overline{\Omega})$, automatically satisfies the homogeneous Dirichlet boundary condition. For simple domains such as a ball or a rectangle, the desired function $w$ can be easily constructed. For complex geometries, the construction is usually non-trivial, and we refer to \cite{sukumar2022exact} for an approach based on distance fields.

We now construct a constrained neural parameterization for \eqref{equa:poly_boundary}. Let $\{\hat{\alpha}_i\}_{i=1}^m$, with $m\le n$, be an orthonormal basis of $\operatorname{Span}\{\alpha_1,\dots,\alpha_n\}$ in $Y$. Then, by Lemma~\ref{lemma:orthocond}, there exist $\{\beta_i\}_{i=1}^m\subset X_0$ such that $\langle \hat{\alpha}_i,\beta_j\rangle_Y=\delta_{ij}$, $i,j=1,\dots,m$. Let $\mathcal{N}(\cdot,\theta)$ be a neural network with trainable parameter set $\theta$, and let $\{\lambda_i\}_{i=1}^p$ and $\{\gamma_i\}_{i=1}^q$ be trainable scalars. Let  $X_\beta:=\operatorname{Span}\{\beta_1,\dots,\beta_m\}$ and define the CNP as
\begin{equation}\label{equa:CNPboundpoly}
    U(\cdot,\hat{\theta})
    =
    \mathcal{N}(\cdot,\theta)\,w
    -
    \sum_{i=1}^m \langle \mathcal{N}(\cdot,\theta)w,\hat{\alpha}_i\rangle_Y \beta_i
    +
    \sum_{i=1}^p g_i(\lambda_1,\dots,\lambda_p)\,y_i
    +
    \sum_{i=1}^q h_i(\gamma_i)\,z_i, \quad \hat{\theta} = \{\theta, \{\lambda_i\},\{\gamma_i\}\},
\end{equation}
where the functions $\{g_i\}_{i=1}^p$ and $\{h_i\}_{i=1}^q$ satisfy Condition~\ref{conditions}, and $\{y_i\}_{i=1}^p$ and $\{z_i\}_{i=1}^q$ denote the vertices and extreme rays, respectively, of the pointed polyhedron $\{u\in X_\beta\mid \langle \alpha_i,u\rangle_Y\le b_i,\ i=1,\dots,n\}$.

\subsubsection*{Universal approximation for polyhedra with Dirichlet boundary conditions}

If the family $\{\mathcal{N}(\cdot,\theta)w:\theta\}$ is dense in $X_0$, then the family $\{U(\cdot,\hat{\theta}):\hat{\theta}\}$ defined by \eqref{equa:CNPboundpoly} is dense in $\mathcal{P}^0$. We next provide a sufficient condition ensuring this density property.

\begin{lemma}[Density of \(\{\mathcal N w\}\) in \(H_0^k(\Omega)\)]
\label{lem:density_wN_X0}
Let $\Omega\subset\mathbb{R}^d$ be a smooth bounded domain, and let $w\in C^k(\overline{\Omega})\cap H_0^k(\Omega)$ satisfy \eqref{equa:boundry}. Let the neural network family $\{\mathcal N(\cdot,\theta):\theta\}$ be dense in $H^k(\Omega)$, then the set $\{\mathcal N(\cdot,\theta)\,w:\theta\}$ is dense in \(H_0^k(\Omega)\) with respect to the \(H^k\)-norm.
\end{lemma}

\begin{proof}
We first show that \(wC^k(\overline{\Omega})\) is dense in \(H_0^k(\Omega)\). Let \(\phi\in C_c^\infty(\Omega)\). Choose an open set \(U\) such that $\operatorname{supp}(\phi)\Subset U\Subset \Omega$. Since $w\in C^k(\overline{\Omega})$ and $w>0$ in \(\Omega\), there exists $c>0$ such that $w\ge c\text{ on } \overline{U}$, and therefore \(1/w\in C^k(\overline{U})\). Define
\[
v(x):=
\begin{cases}
\phi(x)/w(x), & x\in U,\\
0, & x\in \Omega\setminus U.
\end{cases}
\]
Then $v\in C^k(\overline{\Omega})$ and $wv=\phi  \text{ in } \Omega$. Hence $C_c^\infty(\Omega)\subset wC^k(\overline{\Omega})$.

Next, since $w\in H_0^k(\Omega)$ and $v\in C^k(\overline{\Omega})$, it follows that $wv\in H_0^k(\Omega)$. Therefore, $wC^k(\overline{\Omega})\subset H_0^k(\Omega)$ and
\[
H_0^k(\Omega)
=
\overline{C_c^\infty(\Omega)}^{\|\cdot\|_{H^k}}
\subset
\overline{wC^k(\overline{\Omega})}^{\|\cdot\|_{H^k}}
\subset
H_0^k(\Omega).
\]
Hence $\overline{wC^k(\overline{\Omega})}^{\|\cdot\|_{H^k}}
=H_0^k(\Omega)$.

Finally, let \(\bar u\in H_0^k(\Omega)\) and \(\varepsilon>0\). Since $w\in C^k(\overline{\Omega})$, there exists a constant $C>0$ such that $\|w\psi\|_{H^k(\Omega)}\le C\|\psi\|_{H^k(\Omega)},\forall \psi\in H^k(\Omega).$ By the density just proved, there exists \(\bar v\in C^k(\overline{\Omega})\) such that $\|\bar u-w\bar v\|_{H^k}<\varepsilon/2.$ Since the neural network family is dense in \(H^k(\Omega)\), there exists \(\theta\) such that $\|\bar v-\mathcal N(\cdot,\theta)\|_{H^k(\Omega)}
<
\varepsilon/(2C)$. Thus, $
\|w\bar v-w\mathcal N(\cdot,\theta)\|_{H^k}
\le
C\|\bar v-\mathcal N(\cdot,\theta)\|_{H^k}
<\varepsilon/2. $
Therefore, $\|\bar u-w\mathcal N(\cdot,\theta)\|_{H^k}<\varepsilon.$ This proves the density of $\{\mathcal N(\cdot,\theta)w\}$ in $H_0^k(\Omega)$.
\end{proof}

\section{Pointwise constraints}
\label{sec:pointwise}

In this section, we present constrained neural parameterizations for several
classes of pointwise constraints. We emphasize that approximation properties for pointwise constrained sets in Sobolev spaces are more delicate than those for the polyhedral constraints considered in Section~\ref{sec:polyhedral}. In particular, the density of a smoother class in a Sobolev space does not, in general, imply the density of the corresponding pointwise constrained feasible sets. Such density properties may
depend on the topology of the underlying function space, the specific form of
the pointwise constraint, the regularity of the bound or obstacle, and the
presence of boundary conditions. Density results and counterexamples for closed
convex sets with pointwise constraints in Sobolev spaces were studied in
\cite{hintermuller2015density}. A general framework for density of convex
intersections and its role in regularization, discretization, dualization,
variational inequalities, and finite-element approximations was developed in
\cite{hintermuller2017density}. In the present work, we focus on the
construction of smooth neural parameterizations that satisfy the pointwise
constraints exactly. Whenever density is invoked for such constraints, it is to
be understood under suitable feasible-density and regularity assumptions in the
chosen topology. A detailed analysis of sharp density conditions for pointwise
Sobolev constraints is beyond the scope of this paper.

\subsection{Unilateral pointwise inequality constraints}
\label{sec:bound_const}

Let $X=H^k(\Omega)$ be a Sobolev space defined on a bounded domain
$\Omega\subset\mathbb R^d$, and let $\phi\in X$. We consider the unilateral pointwise inequality constraint
\begin{equation}
    \mathcal P
    =
    \{u\in X\mid u(x)\ge \phi(x)\ \text{a.e. in }\Omega\}.
\end{equation}
Assume that the network $\mathcal N(\cdot,\theta)$ is chosen such that
$\mathcal N(\cdot,\theta)^2\in X$ for every parameter vector $\theta$. We define
\begin{equation}
    U(\cdot,\theta)
    =
    \phi+\mathcal N(\cdot,\theta)^2 .
    \label{equa:CNP_lowbound}
\end{equation}
Then $U(\cdot,\theta)\in\mathcal P$. 

We next consider a unilateral pointwise constraint together with a Dirichlet
boundary condition:
\begin{equation}
    \mathcal P^b
    =
    \{u\in X\mid u\ge\phi \ \text{a.e. in }\Omega,\quad
    u=g \ \text{on }\partial\Omega\}.
\end{equation}
Choose a lifting $\tilde g$ satisfying $\tilde g=g \text{ on }\partial\Omega,$ and $\tilde g\ge\phi\,\text{ a.e. in }\Omega$. Let $w$ satisfy \eqref{equa:boundry}.
We define
\begin{equation}
    U_b(\cdot,\theta)
    =
    \phi+
    \left(
        w\,\mathcal N(\cdot,\theta)
        +
        (\tilde g-\phi)^{1/2}
    \right)^2 .
    \label{equa:CNP_lowbound_boundary}
\end{equation}
Under suitable regularity assumptions, one has $U_b(\cdot,\theta)\in\mathcal P^b$ for every parameter vector $\theta$. In the special case $g=\phi|_{\partial\Omega}$, a simpler construction is $U_b(\cdot,\theta)
    =
    \phi+w\,\mathcal N(\cdot,\theta)^2;$ see also \cite{song2026single}.

\subsection{Bilateral pointwise constraints}

Let $\phi_1,\phi_2\in X$, with $X$ as above, satisfy $\phi_1\le \phi_2$, and consider the bilaterally constrained set
\begin{equation}
    \mathcal{P}
    =
    \{u\in X\mid \phi_1\le u\le \phi_2 \text{ a.e. in } \Omega\}.
\end{equation}
A CNP for this constraint set was proposed in \cite{song2026single}:
\begin{equation}
    U(\cdot,\theta)
    =
    \phi_2-\operatorname{ReLU}\bigl(\phi_2-[\operatorname{ReLU}(\mathcal{N}(\cdot,\theta)-\phi_1)+\phi_1]\bigr),
\end{equation}
where $\operatorname{ReLU}(\phi)(x)=\max\{0,\phi(x)\}$, $x\in\Omega$, for a given function $\phi\in X$. This construction provides a shortcut form of the clipping operator. 
The inner ReLU enforces the lower bound \(\phi_1\), while the outer ReLU enforces the upper bound \(\phi_2\). Equivalently, $
U(\cdot,\theta)
=
\min\left\{
\max\left\{\mathcal{N}(\cdot,\theta),\phi_1\right\},
\phi_2
\right\}$. However, ReLU-based constructions provide only limited regularity. This is undesirable in problems where the objective functional or the constraints involve high-order derivatives of $u$. We therefore propose the following smooth alternative:
\begin{equation}\label{mh.bilateral}
    U(\cdot,\theta)
    =
    \phi_1+\sin^2(\mathcal{N}(\cdot,\theta))(\phi_2-\phi_1).
\end{equation}
We also consider the case with an additional Dirichlet boundary condition
\begin{equation}
    \mathcal{P}^b
    =
    \{u\in X\mid \phi_1\le u\le \phi_2 \text{ a.e. in } \Omega,\ u=g \text{ on }\partial\Omega\},
    \qquad
    \phi_1\le g\le \phi_2 \text{ on } \partial\Omega.
\end{equation}
Using the same boundary-lifting function $w$ as in the previous section, we define
\begin{equation}
    U_b(\cdot,\theta)
    =
    \phi_1+\sin^2(\mathcal{N}(\cdot,\theta)w+\hat g)(\phi_2-\phi_1),
\end{equation}
where $\hat g$ is chosen such that
%\begin{equation}
    $\sin^2(\hat g)(\phi_2-\phi_1)=g-\phi_1
    \text{ on } \partial\Omega$.
%\end{equation}
Then the bilateral constraint is automatically enforced, and since $w=0$ on $\partial\Omega$, the boundary condition is satisfied as well. For instance, one may take $\hat g\equiv 0$ if $g=\phi_1$ on $\partial\Omega$, or $\hat g\equiv \pi/2$ if $g=\phi_2$ on $\partial\Omega$.

\subsection{Constraints on multiple phases}

Many multiphysics problems involve variables representing volume fractions or phase concentrations, which are required to satisfy both nonnegativity and a summation constraint. Consider the set
\begin{equation}
    \mathcal{P}
    =
    \bigg\{
    u=(u_1,\dots,u_m)\in X^m
    \,\bigg|\,
    \sum_{i=1}^m u_i(x)=1,\quad
    u_i(x)\ge 0 \text{ a.e. in } \Omega,\ i=1,\dots,m
    \bigg\},
\end{equation}
which is sometimes referred to as the Gibbs simplex.
To construct a CNP for this constraint set, we introduce $m-1$ neural networks $\mathcal{N}(\cdot,\theta_i)$, $i=1,\dots,m-1$, and define
\begin{equation}
\begin{cases}
     u_1 = \sin^2(\mathcal{N}(\cdot,\theta_1)),\\[4pt]
     u_j = \left(1-\sum_{i=1}^{j-1}u_i\right)\sin^2(\mathcal{N}(\cdot,\theta_j)),
     \quad j=2,\dots,m-1,\\[4pt]
     u_m = 1-\sum_{i=1}^{m-1}u_i.
\end{cases}
\end{equation}

In the presence of additional homogeneous Dirichlet boundary conditions yielding
\begin{equation}
    \widetilde{\mathcal{P}}
    =
    \bigg\{
    u=(u_1,\dots,u_m)\in X^m
    \,\bigg|\,
    \sum_{i=1}^m u_i(x)\le 1,\quad
    u_i(x)\ge 0 \text{ a.e. in } \Omega,\quad
    u_i=0 \text{ on }\partial\Omega,\ i=1,\dots,m
    \bigg\},
\end{equation}
we introduce $m$ neural networks $\mathcal{N}(\cdot,\theta_i)$, $i=1,\dots,m$, and define
\begin{equation}
\begin{cases}
     u_1 = \sin^2(\mathcal{N}(\cdot,\theta_1)w(x)),\\[4pt]
     u_j = \left(1-\sum_{i=1}^{j-1}u_i\right)\sin^2(\mathcal{N}(\cdot,\theta_j)w(x)),
     \quad j=2,\dots,m.
\end{cases}
\end{equation}

\section{PDE-constrained optimization}
\label{sec:PDE}

Now we turn to optimization problems with PDE constraints. 
In Subsection~\ref{sec:pde_reduced}, we introduce a constrained neural parameterization for a class of PDEs admitting a separable structure, which leads to an exact reduced neural method. By combining this construction with the CNP schemes developed in the previous sections, one can handle PDE-constrained optimization problems with additional state constraints in a natural way. Moreover, in Subsection~\ref{sec:singularity}, we propose a singularity-enriched exact reduced neural method for Poisson problems with low-regularity states in polygonal domains.

\subsection{Exact reduced neural method}\label{sec:pde_reduced}

Let $\Omega\subset\mathbb{R}^d$ be a bounded domain with smooth boundary $\partial\Omega$. Let $Y$ and $U$ denote the state and control spaces, respectively, and consider the optimal control problem
\begin{align*}
    &\text{minimize}\quad J(y,u)\qquad\text{over } y\in Y, u\in U,\\
%\end{equation}
    &\text{subject to}\quad
%subject to the PDE constraint
%\begin{equation}
    \mathbf{E}(y,u)=0,\\
    %\quad
    &\phantom{\text{subject to}}\quad y\in Y_{ad},
\end{align*}
where $y$ and $u$ denote the state and control, respectively, $J:Y\times U\to\mathbb{R}$ is the objective functional, $\mathbf{E}:Y\times U\to Z$ is the underlying PDE with some Banach space $Z$, and $Y_{ad}\subset Y$ is a non-empty closed convex set. 

Most existing neural network approaches introduce two networks, say $y(\cdot,\theta_y)$ and $u(\cdot,\theta_u)$, and enforce the PDE constraint by penalizing the PDE residual $\bigl\|\mathbf{E}\bigl(y(\cdot,\theta_y),u(\cdot,\theta_u)\bigr)\bigr\|_{L^2(\Omega)}$, assuming sufficient regularity to justify the $L^2(\Omega)$-norm. In contrast, for PDEs admitting a specific separable structure, we show that only one neural network is needed and that the PDE constraint can be built into the parameterization exactly. 

Assume that the PDE constraint allows one to isolate $u$ as a function of $y$, i.e., there exists an operator $\mathbf{G}:Y\to U$ such that $u=\mathbf{G}(y)$ for all $y\in Y_{ad}$.
%\begin{equation}
%    u=\mathcal{L}(y),
%    \qquad
%    y\in Y_{ad},
%\end{equation}
%where $\mathcal{L}$ is a differential operator or an operator that can be evaluated directly under a regular assumption on the state $y$. 
The original optimal control problem then reduces to
\begin{align*}
    &\text{minimize}\quad J\bigl(y,\mathbf{G}(y)\bigr)\qquad \text{over } y\in Y,\\
    &\text{subject to}\quad y\in Y_{ad}. 
\end{align*}
Suppose further that a constrained neural parameterization is available for the admissible state set $Y_{ad}$. As a simple example, consider the optimal control problem with a second-order linear elliptic PDE constraint and a pointwise state constraint:
\begin{equation}
    -\Delta y=f+u
    \quad \text{in }\Omega,
    \qquad
    y\in H^1_0(\Omega), u\in L^2(\Omega)
    ,
    \qquad
    y\le 1
    \quad \text{in }\Omega,
\end{equation}
and $f\in L^2(\Omega)$ given. Then we have $u=\mathbf{G}(y)=-\Delta y - f$,
%\[
%u=\mathcal{L}(y):=-\Delta y-f 
%\]
and the admissible state set is $Y_{ad}=\{y\in H_0^1(\Omega)\mid y\le 1 \text{ in }\Omega\}$. Using the CNP scheme from Subsection~\ref{sec:bound_const} and assuming that the optimal state satisfies the regularity assumption $y^*\in H^2(\Omega)\cap H_0^1(\Omega)$, we parameterize the state and control by
\begin{equation*}
    y(\cdot,\theta)=1-\bigl(\mathcal N(\cdot,\theta)w+1\bigr)^2,
    \qquad
    u(\cdot,\theta)=\mathbf{G}(y(\cdot,\theta))=-\Delta y(\cdot,\theta)-f,
\end{equation*}
where $\mathcal N(\cdot,\theta)$ denotes a neural network with parameter vector $\theta\in \mathbb{R}^M$, and $w$ is the boundary-lifting function introduced in \eqref{equa:boundry}. The  discrete training problem is then reduced to minimizing
\begin{equation}
    \text{minimize}\quad\mathrm{Loss}(\theta):=J\bigl(y(\cdot,\theta),\mathbf{G}(y(\cdot,\theta))\bigr)\qquad\text{over}\quad\theta \in\mathbb{R}^M.
\end{equation}
Note that the required derivatives of $y(\cdot,\theta)$ can be evaluated by automatic differentiation. The PDE constraint is enforced exactly at every collocation or evaluation point, rather than approximately through a residual penalty. Moreover, only one network is introduced, which may reduce computational cost and memory usage compared with approaches that require separate networks for the state and control, and possibly additional networks for adjoint variables or multipliers. Another important feature of the exact reduced neural method is that it avoids the explicit approximation of the adjoint state and Lagrange multipliers. In PDE-constrained optimization problems with pointwise state constraints, the associated multipliers may exhibit very low regularity or even belong to spaces of measures rather than $L^p$  \cite{hintermuller2006feasible,troltzsch2010optimal}, which makes them difficult to approximate accurately by neural networks. By optimizing directly over parameterized admissible states and computing the control via the relation $u=\mathbf{G}(y)$, the proposed formulation bypasses this difficulty. 

\begin{remark} We realize that due to more general applicability eliminating the state $y$ from the unknowns via an application of the implicit function theorem such that $y(u)=\mathbf{S}(u)$, with $\mathbf{S}:U\to Y$, such that $\mathbf{E}(y(u),u)=0$ is perhaps more common than eliminating the control. However, in such a setting the application of the CNP principle to ensure the state constraint becomes problematic, in general.
\end{remark}

The above approach can be readily carried over to PDE systems yielding
%More generally, consider a system with controls $u_1,\dots,u_m$ and states $y_1,\dots,y_n$. Assume that the PDE constraints and state constraints can be reformulated as
\begin{equation}\label{equa:multi-pde}
    \begin{cases}
        u_i=\mathbf{G}_i(y_1,\dots,y_n), & i=1,\dots,m,\\
        y_i\in Y_{ad}^i, & i=1,\dots,n,
    \end{cases}
\end{equation}
with $\mathbf{G}_i:Y_1,\ldots\times Y_n\to U_i$ and Banach spaces $Y_i, U_j$, and nonempty, closed, convex $Y_{ad}^i\subset Y_i$ for $i=1,\ldots,n$ and $j=1,\ldots,m$. If CNP schemes
%\[
$y_i(\cdot,\hat\theta_{y_i})\in Y_{ad}^i$,
%\qquad
$i=1,\dots,n$,
%\]
are available, then under suitable regularity assumptions on the optimal states, analogous to the one above, the constrained minimization problem
\begin{equation}
    \min J(y_1,\dots,y_n,u_1,\dots,u_m)
    \qquad
    \text{subject to \eqref{equa:multi-pde}}
\end{equation}
can be solved by considering the unconstrained optimization of
\begin{equation}
    \mathrm{Loss}(\hat\theta_{y_1},\dots,\hat\theta_{y_n})
    =
    J\bigl(
    y_1(\cdot,\hat\theta_{y_1}),
    \dots,
    y_n(\cdot,\hat\theta_{y_n}),
    \mathbf{G}_1^\theta,
    \dots,
    \mathbf{G}_m^\theta
    \bigr),
\end{equation}
where
%\[
$\mathbf{G}_i^\theta
:=
\mathbf{G}_i\bigl(
y_1(\cdot,\hat\theta_{y_1}),
\dots,
y_n(\cdot,\hat\theta_{y_n})
\bigr)$,
$i=1,\dots,m$
.

For notational simplicity, we discuss the approximation property only for a single
state-control pair; the extension to multiple states and controls is analogous.
Assume that, for some integer $k$, the reduced operator satisfies
$\mathbf{G}:Y_{ad}\cap H^k(\Omega)\to U$ and is continuous when
$Y_{ad}\cap H^k(\Omega)$ is endowed with the $H^k(\Omega)$-topology. That is, for
any sequence $\{y_n\}_{n\in\mathbb N}\subset Y_{ad}\cap H^k(\Omega)$,
$y_n\to y$ in $H^k(\Omega)$ implies $\mathbf{G}(y_n)\to \mathbf{G}(y)$ in $U$.

Assume further that the corresponding CNP ansatz family is dense in
$Y_{ad}\cap H^k(\Omega)$ with respect to the $H^k(\Omega)$-norm; namely, for
every $y\in Y_{ad}\cap H^k(\Omega)$, there exists a sequence of parameters
$\{\hat\theta_n\}_{n\in\mathbb N}$ such that
$y(\cdot,\hat\theta_n)\to y$ in $H^k(\Omega)$. Then, by the continuity of
$\mathbf{G}$, the associated reconstructed controls
$u(\cdot,\hat\theta_n):=\mathbf{G}(y(\cdot,\hat\theta_n))$ satisfy
$u(\cdot,\hat\theta_n)\to \mathbf{G}(y)=u$ in $U$. Thus the parameterized pairs
$\bigl(y(\cdot,\hat\theta_n),u(\cdot,\hat\theta_n)\bigr)$ approximate admissible
state-control pairs in the topology $H^k(\Omega)\times U$. In particular, any
optimal pair $(y^\ast,u^\ast)$ satisfying the above regularity and the reduced
relation $u^\ast=\mathbf{G}(y^\ast)$ can be approximated in this sense.

The above assumptions are nontrivial. For example, for the Poisson equation
$-\Delta y=f+u$ in $\Omega$, $y=0$ on $\partial\Omega$, with $U=L^2(\Omega)$, the
reduced operator is $\mathbf{G}(y)=-\Delta y-f$. Then
$y\in H^2(\Omega)\cap H^1_0(\Omega)$ and $f\in L^2(\Omega)$ are sufficient to
ensure $\mathbf{G}(y)\in L^2(\Omega)$ and the continuity of
$\mathbf{G}:H^2(\Omega)\cap H^1_0(\Omega)\to L^2(\Omega)$. In contrast, for a
weakly feasible pair $(y,u)\in H^1_0(\Omega)\times L^2(\Omega)$, the state $y$
need not belong to $H^2(\Omega)$, for instance when elliptic $H^2$-regularity
fails on nonsmooth or nonconvex domains~\cite{dauge2006elliptic}. Hence one cannot generally expect a
sequence $\{y_n\}_{n\in\mathbb N}\subset H^2(\Omega)\cap H^1_0(\Omega)$ such
that $y_n\to y$ in $H^1_0(\Omega)$ and $-\Delta y_n-f\to u$ in $L^2(\Omega)$.

In the next subsection, we explore the use of singularity-enriched neural methods for optimal control problems governed by the Poisson equation in polygonal domains with geometric singularities, where the regularity assumption on the optimal state cannot be guaranteed.

\subsection{Singularity-enriched method for Poisson problems in polygonal domains}\label{sec:singularity}

Given a polygonal domain $\Omega\subset\mathbb{R}^2$ with a re-entrant corner, $f\in L^2(\Omega)$ and $\alpha>0$, we consider 
\begin{align*}
    &\text{minimize}\quad J(y,u)=\frac{1}{2}\|y-y_d\|_{L^2(\Omega)}^2 + \frac{\alpha}{2}\|u\| _{L^2(\Omega)}^2\qquad \text{over }(y,u)\in H^1_0(\Omega)\times L^2(\Omega),\\
&\text{subject to}\quad -\Delta y = u+f  \text{ in } \Omega, 
\qquad y = 0  \text{ on } \partial \Omega.
\end{align*}
It is known that re-entrant corners induce corner singularities in the solution,
so that, in general, $y\notin H^2(\Omega)$~\cite{dauge2006elliptic}. As a consequence, a direct application of strong-form neural approaches, such as PINNs, makes it difficult to capture the singular component of the solution. To address this issue, inspired by the singularity-enriched PINNs\cite{hu2024solving}, we exploit the analytical structure of the solution and decompose it as
%\begin{equation}
$y = \mathcal{R} + \mathcal{S}$,
%\end{equation}
where $\mathcal{R} \in H^2(\Omega)$ is a smooth regular component and $\mathcal{S}$ captures the singular behavior near corners.

For Dirichlet boundary conditions, as above, our choice of singular functions takes the form
%\begin{equation}
$s_j(r_j,\theta_j) = r_j^{\lambda_j} \sin(\lambda_j \theta_j)$, 
%\qquad 
$\lambda_j = \frac{\pi}{\omega_j}$,
%\end{equation}
where $\omega_j$ is the interior angle at vertex $v_j$, $j=1,\ldots, M$, and $(r_j,\theta_j)$ are local polar coordinates. The singular part $\mathcal{S}$ is then constructed as
\begin{equation}
\mathcal{S} = \sum_{j=1}^M \gamma_j \, \eta_\rho(r_j)\, s_j(r_j,\theta_j),
\end{equation}
where $\eta_\rho\in C^2([0,+\infty))$, defined for $\rho \in(0,2]$, is a smooth cutoff function ensuring localization, and $\{\gamma_j\}$ are unknown coefficients. A typical choice of $\eta_\rho$ is 
\begin{equation}\label{equa:rho}
\eta_\rho(r)=
\begin{cases}
1, & 0<r<\frac{\rho L}{2},\\
\frac{15}{16}\bigg(\frac{8}{15}-\big(\frac{4r}{\rho L}-3\big)+\frac{2}{3}\big(\frac{4r}{\rho L}-3\big)^3-\frac{1}{5}\big(\frac{4r}{\rho L}-3\big)^5\bigg) ,
& \frac{\rho L}{2}\le r\le \rho L,\\
0, & \rho L\le r,
\end{cases}
\end{equation}
where $L\in \mathbb{R}_{+}$ is a fixed small number so that $\eta_\rho(r_j)s_j(r_j,\theta_j)$ preserves the zero boundary condition on $\partial \Omega$.

Substituting the decomposition into the PDE yields a modified equation for the regular part $\mathcal{R}$:
\begin{equation}
-\Delta \mathcal{R} -\sum_{j=1}^M \gamma_j \Delta(\eta_\rho s_j)= u+f
\quad \text{in } \Omega,
\qquad \mathcal{R} = 0 \quad \text{on } \partial \Omega.
\end{equation}
Since $s_j$ is harmonic away from the corner singularity, one has
%\begin{equation}
    $\Delta(\eta_\rho s_j)=\Delta\big((\eta_\rho-1) s_j \big) + \Delta s_j = \Delta\big((\eta_\rho-1) s_j \big)$
%\end{equation}
where $(\eta_\rho-1) s_j\in C^2(\overline{\Omega
})$ by construction. Thus, $\Delta(\eta_\rho s_j)$ can be explicitly evaluated. We now introduce a neural network $\mathcal{N}(\cdot,\theta_R)$ and approximate $\mathcal{R}$ using $\mathcal{R}(\cdot,\theta_R):=\mathcal{N}(\cdot,\theta_R)w$ with $w$ satisfying \eqref{equa:boundry}, and treat the coefficients $\{\gamma_j\}$ as trainable parameters. The parameters are obtained by minimizing the loss function
\begin{equation}
    \text{Loss}(\theta_R,\{\gamma_j\}):= J\big(\mathcal{R}(\cdot,\theta_R) +\sum_{j=1}^M \gamma_j (\eta_\rho s_j), -\Delta \mathcal{R}(\cdot,\theta_R)-  \sum_{j=1}^M \gamma_j \Delta(\eta_\rho s_j)-f\big).
\end{equation}
Once the training is completed, we approximate the optimal state and control as
\begin{equation}
    y_\theta=\mathcal{R}(\cdot,\theta_R) +\sum_{j=1}^M \gamma_j ( \eta_\rho s_j),\quad u_\theta=-\Delta \mathcal{R}(\cdot,\theta_R)-  \sum_{j=1}^M \gamma_j \Delta(\eta_\rho s_j)-f.
\end{equation}
It is straightforward to verify that $(y_\theta,u_\theta)$ satisfies the PDE constraint exactly in the weak sense. Such singularity-enriched PINNs have also been developed for mixed boundary conditions and 3D Poisson problems in \cite{hu2024solving}, and the corresponding methods for optimal control problems can be similarly extended.

\section{Numerical experiments}
\label{sec:numerical}

In this section, we present numerical experiments to assess the performance of the proposed constrained neural parameterization (CNP) schemes on several representative problems, including high-dimensional PDE-constrained optimization, state-constrained optimal control, and a PDE-based hybrid inverse problem.

For all experiments we employ an $N$-block ResNet architecture. We use $N=2$ blocks for the one-dimensional problem and $N=3$ blocks for the remaining examples. Each block consists of two fully connected layers with a residual connection. The hidden width is set to 64 for all examples except the 10-dimensional example, where we use width 128. All neural networks use the `tanh' activation function and are trained using the Adam optimizer. Additional implementation details are provided in the corresponding examples. The experiments were implemented in PyTorch and executed on a MacBook Pro equipped with an Apple M5 chip and 24 GB unified memory. GPU acceleration was enabled through the Metal Performance Shaders (MPS) backend. The implementation of the proposed methods is publicly available at \url{https://github.com/JianfengNing/CNP_Code}.

\subsection{Optimization with two integral constraints}\label{subsec:twointegral}

In this example we consider the domain $\Omega=(0,1)$ and study the constrained optimization problem
\begin{align*}
    &\text{minimize}\quad J(u)=
    \|u-u_d\|_{L^2(\Omega)}^2 
    + \alpha\|u'\|_{L^2(\Omega)}^{2}, \qquad \text{over } u\in H^1(\Omega),\\
%\end{equation}
%\begin{equation}
  &\text{subject to}\quad   u\in \mathcal{K}:=
    \left\{
    u\in H^1(\Omega)  
    \ \middle|\ 
    \int_\Omega u(x)dx\le c_1,\;
    \int_\Omega x u(x)dx\le c_2
    \right\}.
\end{align*}
The optimal solution $u^*$, the target function $u_d$, and the parameters $\alpha,c_1,c_2$ are chosen as
\begin{equation}
\begin{cases}
u^* = \cos(\pi x) + \cos(6\pi x) + 2,\quad
u_d = -\alpha u^*_{xx} + u^* + 0.25x, \\
\alpha=0.01, 
\quad 
c_1 = 3,
\quad 
c_2 = 1-\frac{2}{\pi^2}.
\end{cases} 
\end{equation}
Following the construction in Section~\ref{sec:alpha_i}, we fix
\begin{equation}
\begin{cases}
\hat{\alpha}_1=\beta_1\equiv 1, 
\quad 
\hat{\alpha}_2=\beta_2=0.5-x,\quad z_1 = -1+2x,
\quad 
z_2=2-3x,\\
y_1 = a +bx, \text{ where $a,b$ solve }
\int_\Omega (a+bx)dx=c_1,
\int_\Omega x(a+bx)dx=c_2.
\end{cases}
\end{equation}
Introducing a network $\mathcal{N}(\cdot,\theta)$ and trainable parameters $\{\gamma_i\}_{i=1}^2$, we construct the CNP as:
\begin{equation}
     u(\cdot,\hat{\theta})
     =
     \mathcal{N}(\cdot,\theta)
     -\sum_{i=1}^2 
     \langle \mathcal{N}(\cdot,\theta),\hat{\alpha}_i\rangle_{L^2(\Omega)}\beta_i
     + y_1
     + \sum_{i=1}^2 \gamma_i^2 z_i,
     \qquad 
     \hat{\theta} = \{\theta,\gamma_1,\gamma_2\}.
\end{equation}
The loss function then reads:
\begin{equation}
    \text{Loss}(\hat{\theta})
    =
    \|u(\cdot,\hat{\theta})-u_d\|_{L^2(\Omega)}^2
    +\alpha\|u'(\cdot,\hat{\theta})\|_{L^2(\Omega)}^{2}.
\end{equation}
As a comparison, we consider a neural network-based penalty method where the solution is parameterized as $u(\cdot,\theta)=\mathcal{N}(\cdot,\theta),$ and the corresponding discrete loss function reads
\begin{equation}
\begin{split}
     \text{Loss}(\hat{\theta})
    =&
    \|u(\cdot,\theta)-u_d\|_{L^2(\Omega)}^2
    +\alpha\|u'(\cdot,\theta)\|_{L^2(\Omega)}^{2} 
    \\
    &+ \beta \bigg(\text{ReLU} \big(\int_\Omega u(x,\theta)dx-c_1\big)^2 + \text{ReLU} \big(\int_\Omega xu(x,\theta)dx-c_2\big)^2 \bigg).
\end{split}
\end{equation}
The specific choices of $\beta$ are reported below with the experiments.
For this one-dimensional problem, Monte Carlo integration is unnecessary. Instead, at each iteration we employ equispaced points $\{x_i\}_{i=1}^{1000+N}$ in $[0,1]$, where $N$ is randomly drawn from a uniform distribution of $\{0,...,200\}$. The stochastic component helps reduce the risk that the integrand appears to be minimized at a fixed set of nodes while the overall functional remains far from its minimum\cite{yu2018deep}. 

For both methods we train for 5000 epochs. The learning rates are initialized to $10^{-3}$ for the network parameters and $10^{-2}$ for $\{\gamma_i\}_{i=1}^2$, and are reduced by a factor of $0.8$ every 200 epochs. The training loss and error histories of the CNP method and the penalty method with different penalty parameters are shown in Fig.~\ref{fig:training_hist}, and the corresponding relative $L^2$ errors are reported in Table~\ref{tab:error_comparison}. The results show that the proposed CNP method achieves significantly higher accuracy for both the solution and its derivative, while also exhibiting faster convergence compared with the penalty method across different choices of the penalty parameter.
\begin{figure}[htbp]
     \centering
    \begin{subfigure}[b]{0.3\textwidth}
        \centering
        \includegraphics[width=\textwidth]{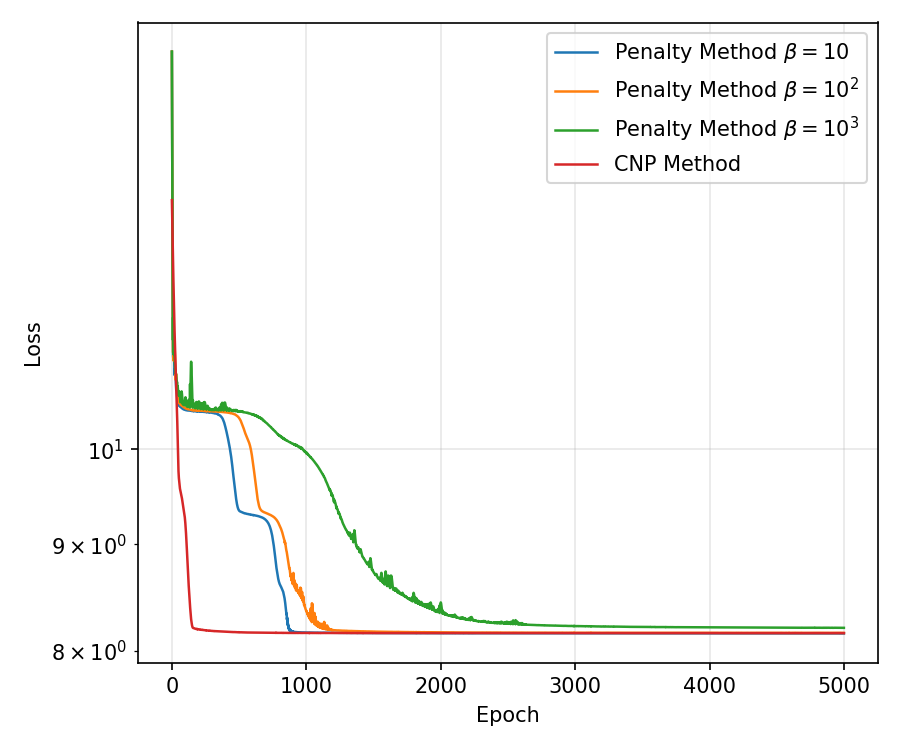}
        \caption{Loss histories.}
    \end{subfigure}
   \hspace{1.3cm}
    \begin{subfigure}[b]{0.3\textwidth}
        \centering
        \includegraphics[width=\textwidth]{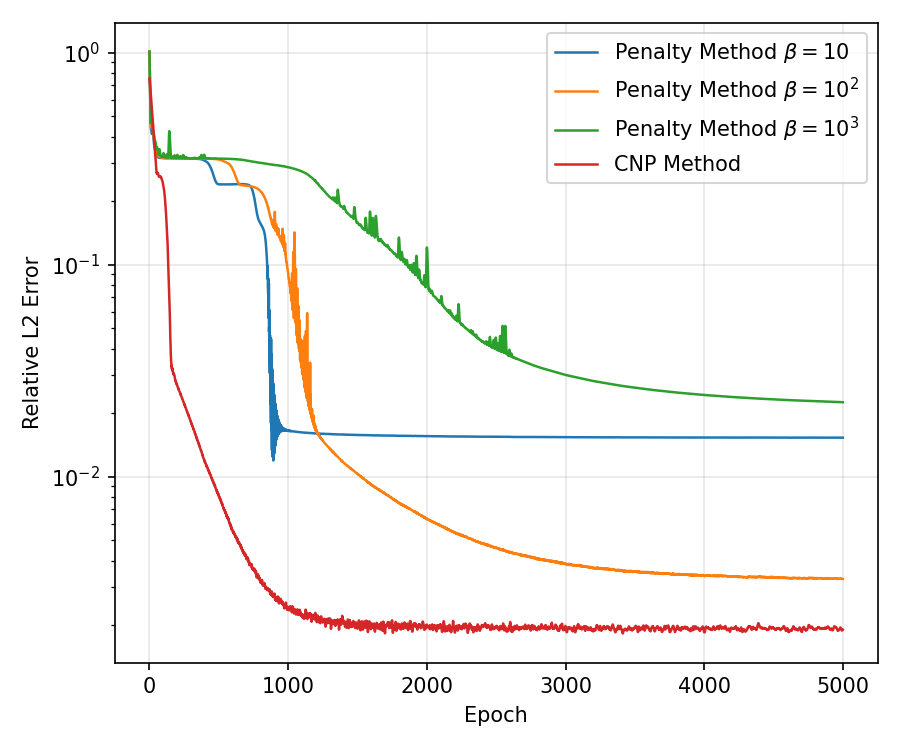}
        \caption{Error histories.}
    \end{subfigure}
    \caption{Numerical results for Example~\ref{subsec:twointegral}. Loss histories and error histories during the training process.}
    \label{fig:training_hist}
\end{figure}
\begin{table}[htbp]
\centering
\begin{tabular}{ccccc}
\toprule
Methods & CNP Method & \makecell{Penalty Method \\ ($\beta=10$)} & \makecell{Penalty Method \\ ($\beta=10^2$)} & \makecell{Penalty Method \\ ($\beta=10^3$)} \\
\midrule
Relative $L^2$ error ($u$) & $1.90 \times 10^{-3}$ & $1.52 \times 10^{-2}$ & $3.29 \times 10^{-3}$ & $2.24 \times 10^{-2}$ \\
Relative $L^2$ error ($u'$) & $6.27 \times 10^{-3}$ & $3.29 \times 10^{-2}$ & $3.80 \times 10^{-2}$ & $1.59 \times 10^{-1}$ \\
\bottomrule
\end{tabular}
\caption{Numerical results for Example~\ref{subsec:twointegral}. Relative $L^2$ errors of the numerical solutions and derivatives.}
\label{tab:error_comparison}
\end{table}

\subsection{A state-constrained optimal control problem in high-dimensional space}
\label{subsec:4D}

In this example we consider the optimal control problem
\begin{align*}
&\text{minimize}\quad J(y,u)
=
\frac{1}{2}\|y-y_d\|^2_{L^2(\Omega)}
+
\frac{\alpha}{2}\|u\|^2_{L^2(\Omega)},
\qquad \text{over }
(y,u)\in H^1_0(\Omega)\times L^2(\Omega),\\
%\end{equation}
%subject to the following PDE constraint with an additional state integral constraint:
%\begin{equation}
&\text{subject to}\quad -\Delta y = u \quad \text{in }\Omega,
\qquad
y=0 \text{ on }\partial\Omega,
\qquad
\int_{\Omega_1} y(x)\,dx \le c .
\end{align*}
We take $\Omega=(-1,1)^{10}, \Omega_1 = (-1,0)\times(-1,1)^9$, and choose the optimal state and control as $y^*(x)
=
\prod_{i=1}^{10} \cos\big(\frac{\pi x_i}{2}\big)$ and $u^*(x)=-\Delta y^*$. We set $c=\frac{1}{2}\left(\frac{4}{\pi}\right)^{10}$, $y_d=-\alpha \Delta u^*+y^*+\frac{1}{2}\chi_{\Omega_1}$. The additional constant term ensures that the state constraint is active and corresponds to a positive Lagrange multiplier in the optimality system.

To construct a constrained neural representation of the state $y$, we introduce a neural network $\mathcal N(\cdot,\theta_y)$, a trainable parameter $\gamma\in\mathbb R$, and the functions $\alpha_1\equiv 1$, $w=\prod_{i=1}^{10}(1- x_i^2)$. Following the CNP construction in Equation~\ref{equa:CNPboundpoly}, we let $\hat{\alpha}_1\equiv 1$ and $\beta_1 = 2(\frac{3}{4})^{10}w$, and define
\begin{equation}
\label{equa:CNP_4d}
y(\cdot,\hat{\theta}_y)
=
\mathcal N(\cdot,\theta_y)w
-
\langle \mathcal N(\cdot,\theta_y)w,\hat\alpha_1\rangle_{L^2(\Omega_1)}\,\beta_1
+
(c-\gamma^2)\beta_1,
\qquad
\hat{\theta}_y=\{\theta_y,\gamma\}.
\end{equation}
The corresponding loss function is
\begin{equation}
\label{equa:4d_loss}
\text{Loss}(\hat{\theta}_y)
=
J\big(y(\cdot,\hat{\theta}_y),-\Delta y(\cdot,\hat{\theta}_y)\big).
\end{equation}
For comparison, we also consider a neural network method based on the KKT system
\begin{equation}
\begin{cases}
-\Delta y = u \quad \text{in }\Omega,
\qquad y=0 \text{ on }\partial\Omega,\\[4pt]
-\Delta p = y-y_d + \mu\chi_{\Omega_1}
\quad \text{in }\Omega,
\qquad p=0 \text{ on }\partial\Omega,\\[4pt]
\alpha u + p = 0 \quad \text{in }\Omega,\\[4pt]
\mu \ge 0,\qquad
\int_{\Omega_1} y(x)dx \le c,\qquad
\mu\!\left(\int_{\Omega_1} y(x)dx-c\right)=0 .
\end{cases}
\end{equation}
In addition to the parameterization $y(\cdot,\hat{\theta}_y)$ in \eqref{equa:CNP_4d}, we introduce a neural network $\mathcal N(\cdot,\theta_p)$ and a trainable parameter $\hat{\mu}$, and set $p(\cdot,\theta_p)=\mathcal N(\cdot,\theta_p)w$. Following \cite{dai2025solving}, we use the relation $\alpha u=-p$ to eliminate $u$ and observe that $\int_{\Omega_1}y(x)dx=c-\gamma^2$, and we obtain the loss with $\Theta=\{\hat{\theta}_y,\theta_p,\hat{\mu}\}$:
\begin{equation}
\text{Loss}(\Theta)
=
\|\Delta y(\cdot,\hat{\theta}_y)-\alpha^{-1}p(\cdot,\theta_p)\|^2_{L^2(\Omega)}
+
\|\Delta p(\cdot,\theta_p)+y(\cdot,\hat{\theta}_y)-y_d+\hat{\mu}^2\chi_{\Omega_1}\|^2_{L^2(\Omega)}
+
(\hat{\mu}\gamma)^2 .
\end{equation}
For both methods we train the networks for 10000 epochs. The learning rates are initialized as $10^{-3}$ for the network weights and $10^{-2}$ for the parameter $\gamma$, and reduced by a factor of $0.8$ every 200 epochs. In each epoch, $2\times10^4$ points are sampled uniformly from $\Omega$ to approximate the integrals in the loss functions. The numerical optimal states and controls on the slice $\{(x_1,x_2,0,...,0)\}$
are shown in Fig.~\ref{fig:PDE_figure_4d}. Both methods provide accurate approximations in this high-dimensional setting, while the exact reduced neural method yields smaller relative $L^2$ and absolute errors at the slice for both state and control. Moreover, since the KKT-based approach requires training of two neural networks, it incurs higher computational cost and memory usage per epoch. We have also tested enforcing the integral constraint using a penalty method instead of the CNP representation \eqref{equa:CNP_4d} for both methods. In that case the numerical solutions exhibit larger errors, and therefore these results are omitted.
\begin{figure}[htbp]
    \centering
    \begin{subfigure}[b]{0.85\textwidth}
        \centering
        \includegraphics[width=\textwidth]{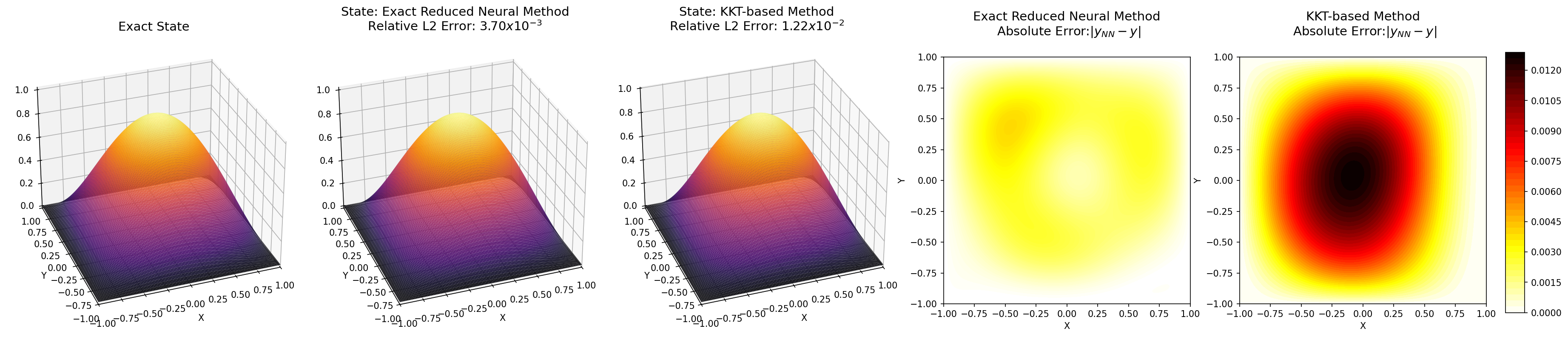}
        \label{fig:PDE_state_figure_4d}
    \end{subfigure}
   \begin{subfigure}[b]{0.85\textwidth}
        \centering
        \includegraphics[width=\textwidth]{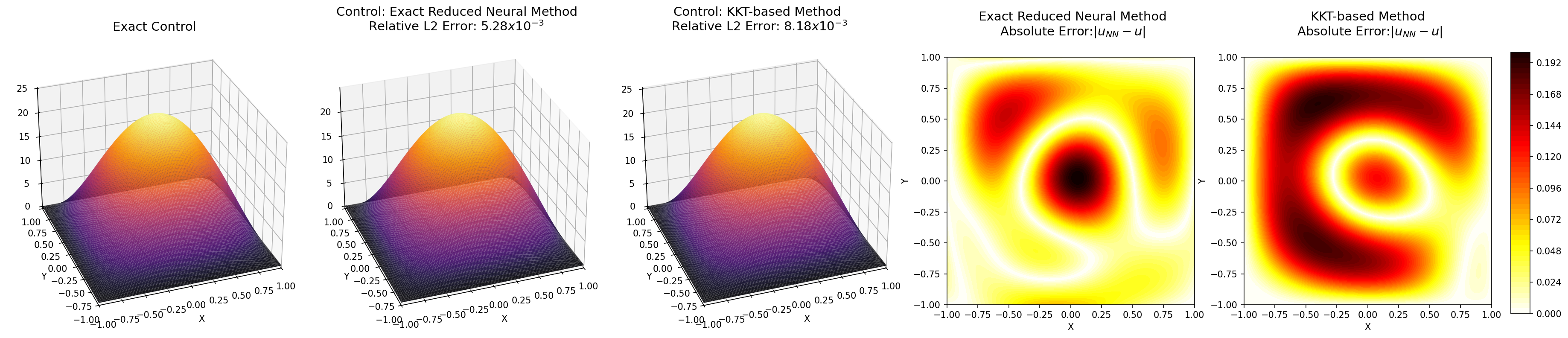}
        \label{fig:PDE_control_figure_highd}
    \end{subfigure}
    \vspace{-0.4cm}
    \caption{Numerical results for Example \ref{subsec:4D}. Numerical optimal states and controls by the exact reduced neural method and the KKT-based neural method at $\{(x_1,x_2,0,...,0)\}$.} 
    \label{fig:PDE_figure_4d}
\end{figure}

\subsection{A state-constrained optimal control problem with low multiplier regularity}
\label{subsec:lowregu}

We consider the following optimal control problem with pointwise state constraints\cite{hintermuller2006feasible}:
\begin{align*}
&\text{minimize}\quad J(y,u)
=
\frac12\|y-y_d\|_{L^2(\Omega)}^2
+
\frac{\alpha}{2}\|u\|_{L^2(\Omega)}^2, 
\qquad \text{over }
(y,u)\in H_0^1(\Omega)\times L^2(\Omega),\\
%\end{equation}
&\text{subject to}\quad
%\begin{equation}
-\Delta y=u \ \text{in }\Omega,
\qquad
y=0 \ \text{on }\partial\Omega,
\qquad
y\le\psi \ \text{a.e. in }\Omega .
\end{align*}
We take $\Omega=(0,1)^2$, $y_d=10(\sin(2\pi x_1)+x_2)$, $\psi=0.01$, and $\alpha=0.1$. This example is challenging because the multiplier associated with the state constraint has low regularity\cite{hintermuller2006feasible}, which makes KKT-based neural approaches challenging to apply. Instead, we employ the exact reduced neural method together with the CNP scheme for the upper-bound constraint and the boundary condition. Introducing a neural network $\mathcal N(\cdot,\theta)$ and the boundary lifting function $w$, we employ the transformation proposed in Section~\ref{sec:bound_const}:
\begin{equation}
    y(\cdot,\theta)
=
\psi-(\mathcal N(\cdot,\theta)w+\sqrt{\psi})^2
=
-\mathcal N(\cdot,\theta)^2 w^2
-
2\sqrt{\psi}\,\mathcal N(\cdot,\theta)w .
\end{equation}
The loss function then reads
\begin{equation}
\label{equa:CNP_low_r}
\text{Loss}(\theta)
=
J\big(y(\cdot,\theta),-\Delta y(\cdot,\theta)\big).
\end{equation}
This formulation avoids solving the KKT system and therefore circumvents the difficulties associated with the low-regularity multiplier. Moreover, no penalty term is introduced.

The networks are trained for 20000 epochs with an initial learning rate of $10^{-3}$, which is reduced by a factor of $0.8$ every 500 epochs. At each epoch the loss function is evaluated on a uniform grid of $(50+N)^2$ points, where $N$ is drawn uniformly from $\{0,\ldots,50\}$.

Since the exact solution is not available in closed form, we compare our results with those obtained by the primal-dual active set method proposed in \cite{hintermuller2006path}. This method solves a regularized version of the problem and converges locally with a q-superlinear rate. Fig.~\ref{fig:state_bound} shows the numerical optimal states obtained by the primal-dual active set method and by the proposed neural approach. We observe that their numerical solutions are highly consistent. In particular, using the primal-dual active set solution as reference, the
relative $L^2$ errors of the exact reduced neural method are
$6.30\times 10^{-4}$ for the state and $1.02\times 10^{-2}$ for the control.

We also attempted a KKT-based neural approach in which $y$, $u$, the adjoint variable, and the multiplier are approximated by four neural networks, and the KKT system is solved using the PINN method with mean square loss. However, due to the low regularity of the multiplier, the KKT-based approach does not yield satisfactory approximations in our experiments, and its numerical results are therefore omitted.
\begin{figure}[htbp]
     \centering
     \begin{subfigure}[b]{0.4\textwidth}
        \centering
        \includegraphics[width=\textwidth]{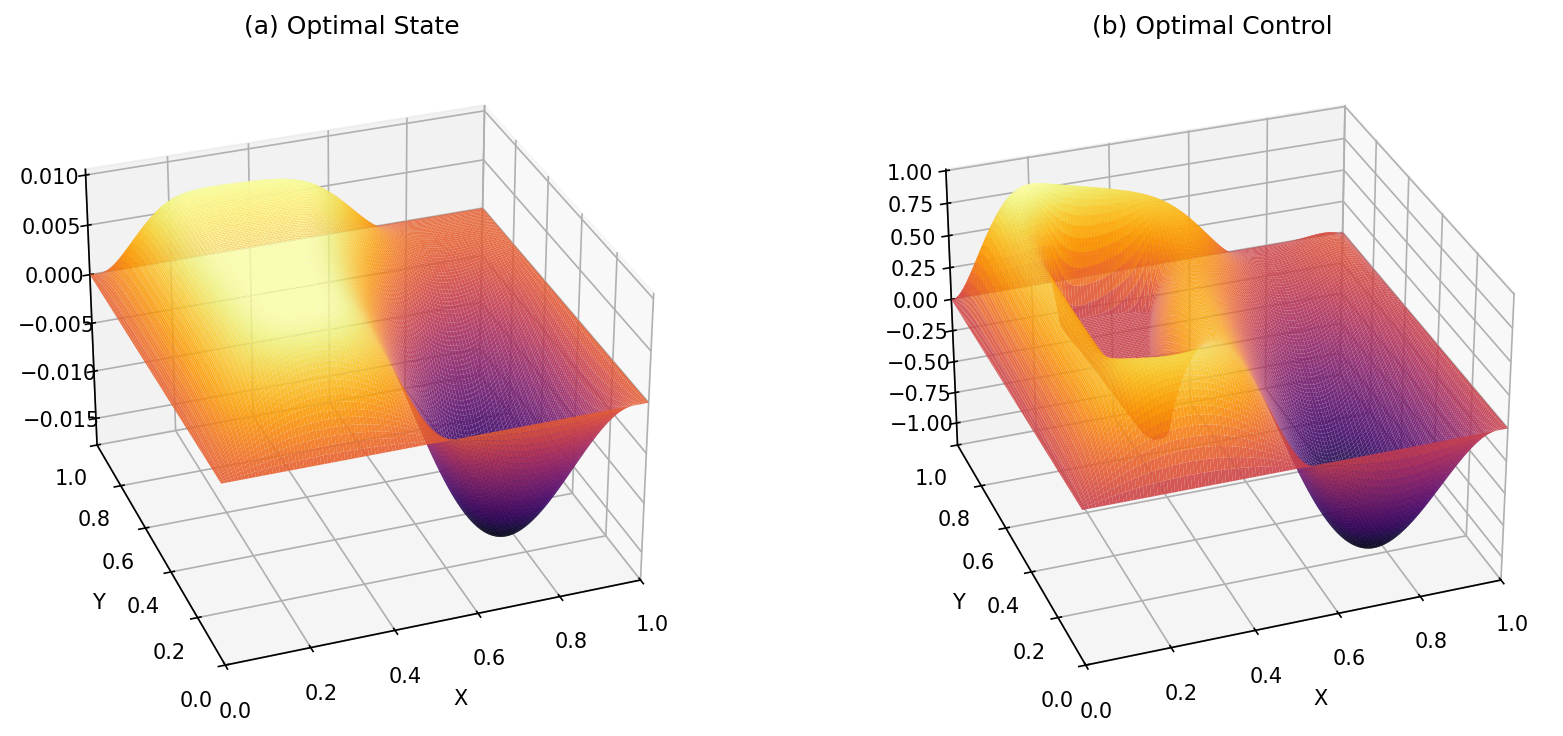}
        \caption{Primal-dual active set method}
    \end{subfigure}
       \hspace{1cm}
    \begin{subfigure}[b]{0.4\textwidth}
        \centering
        \includegraphics[width=\textwidth]{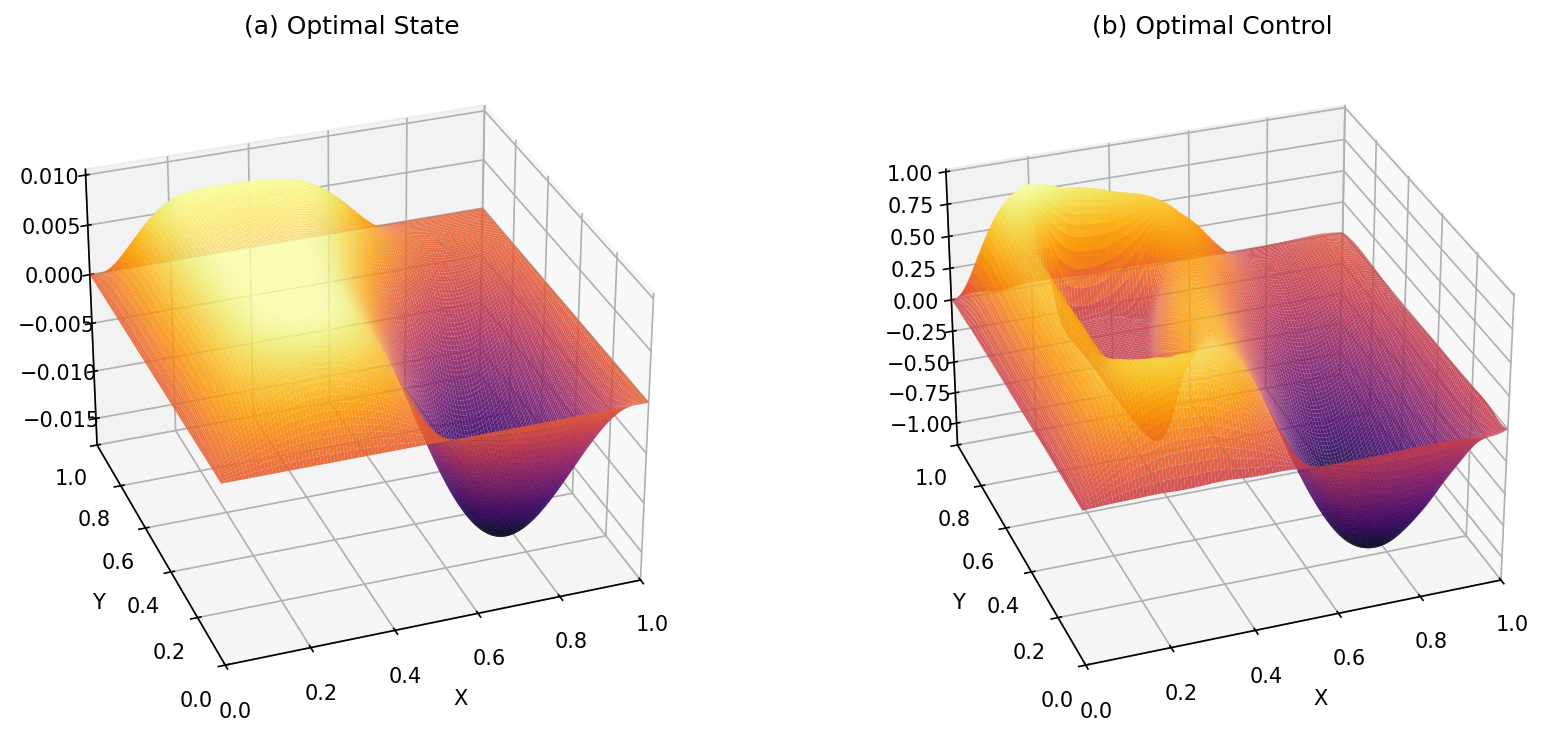}
        \caption{Exact reduced neural method}
    \end{subfigure}
     \caption{Numerical results for Example \ref{subsec:lowregu}. Numerical optimal states and controls computed by the primal-dual active set method and the exact reduced neural method.}
    \label{fig:state_bound}
\end{figure}

\subsection{A mathematical program without strict complementarity}
\label{subsec:mpec}
In this example we consider a mathematical program with complementarity constraints in function space \cite{hintermuller2009mathematical}:
\begin{align}
\nonumber
&\text{minimize}\quad J(y,\xi,u)
=
\frac12\|y-y_d\|_{L^2(\Omega)}^2
+
\frac{\alpha}{2}\|u\|_{L^2(\Omega)}^2,
\qquad
(y,u,\xi)\in H_0^1(\Omega)\times L^2(\Omega)^2,\\
%\end{equation}
&\text{subject to}\quad
%\begin{equation}
\label{equa:const_mp}
\begin{cases}
-\Delta y = u+\xi+f & \text{in }\Omega,\\
y\ge0,\;\xi\ge0 & \text{a.e. in }\Omega,\\
\langle y,\xi\rangle_{L^2(\Omega)}=0.
\end{cases}
\end{align}
We set the exact optimal solutions as
\begin{equation}
   u^*(x_1,x_2)=100 y^*(x_1,x_2),\quad y^*(x_1,x_2)=\begin{cases}
   z_1(x_1)z_2(x_2)\quad \text{in }  (0,0.5)\times(0,0.8),\\
        0\quad \qquad\qquad  \text{else},
    \end{cases}
\end{equation}
\begin{equation}
    \xi^*(x_1,x_2)=50\max(-|x_1-0.8|-|(x_2-0.2)x_1-0.3|+0.35,0).
\end{equation}
Here, we have
\[
z_1(x_1)
=
-4096x_1^6+6144x_1^5-3072x_1^4+512x_1^3,
\quad
z_2(x_2)
=
-244.140625x_2^6+585.9375x_2^5-468.75x_2^4+125x_2^3 .
\]
To satisfy the optimality system, we set $f:=-\Delta y^*-u^*-\xi^*$, and $y_d:=y^*+\xi^*-\alpha\Delta u^*,$ with $\alpha=0.05$. Note that the biactive set $\{(x_1,x_2)\mid y(x_1,x_2)=\xi(x_1,x_2)=0\}$ has positive measure, so strict complementarity does not hold. This situation is particularly challenging because the gradients of the active constraints are linearly dependent at the solution, which prevents the application of the KKT-theory. To implement neural network methods for this problem we introduce two networks $\mathcal N(\cdot,\theta_y)$ and $\mathcal N(\cdot,\theta_\xi)$ and use the CNPs
\begin{equation}
y(\cdot,\theta_y)=\mathcal N(\cdot,\theta_y)^2\,w(x),
\qquad
\xi(\cdot,\theta_\xi)=\mathcal N(\cdot,\theta_\xi)^2,    
\end{equation}
which automatically enforce the nonnegativity constraints and the boundary condition. We employ the exact reduced neural method to satisfy the PDE constraint exactly while enforcing the complementarity condition $\langle y,\xi\rangle_{L^2(\Omega)}=0$ through a penalty term. The resulting loss function then reads
\begin{equation}
\text{Loss}(\theta_y,\theta_\xi)
=
\frac12\|y(\cdot,\theta_y)-y_d\|_{L^2(\Omega)}^2
+
\frac{\alpha}{2}
\|-\Delta y(\cdot,\theta_y)-\xi(\cdot,\theta_\xi)-f\|_{L^2(\Omega)}^2
+
\beta\langle y(\cdot,\theta_y),\xi(\cdot,\theta_\xi)\rangle_{L^2(\Omega)},
\end{equation}
with a penalty parameter $\beta$. Since the parameterizations of $y$ and $\xi$ ensure that the inner product in the penalty term is automatically nonnegative, no additional absolute value or squaring operation is required. In principle one could design a CNP scheme to enforce the orthogonality condition directly, but simultaneously satisfying both nonnegativity and orthogonality constraints in a hard manner is considerably more complicated. By contrast, a pure penalty formulation would require four separate penalty terms to enforce the constraints in \eqref{equa:const_mp}.

We choose $\beta=10^2$, and the networks are trained for 20000 epochs with an initial learning rate of $10^{-3}$, which is reduced by a factor of $0.8$ every 500 epochs. At each epoch, the loss function is evaluated on a uniform grid of $(50+N)^2$ points, where $N$ is drawn uniformly from $\{0,\ldots,50\}$. The exact and numerical optimal states, multipliers, and controls are shown in Fig.~\ref{fig:MPEC}. Even in the absence of strict complementarity, the proposed method produces accurate approximations of the optimal state, multiplier, and control, with relative $L^2$ errors of $1.00\times 10^{-3}, 2.07\times 10^{-2}$, and $1.05\times 10^{-2}$, respectively.
\begin{figure}
    \centering
    \includegraphics[width=0.6\linewidth]{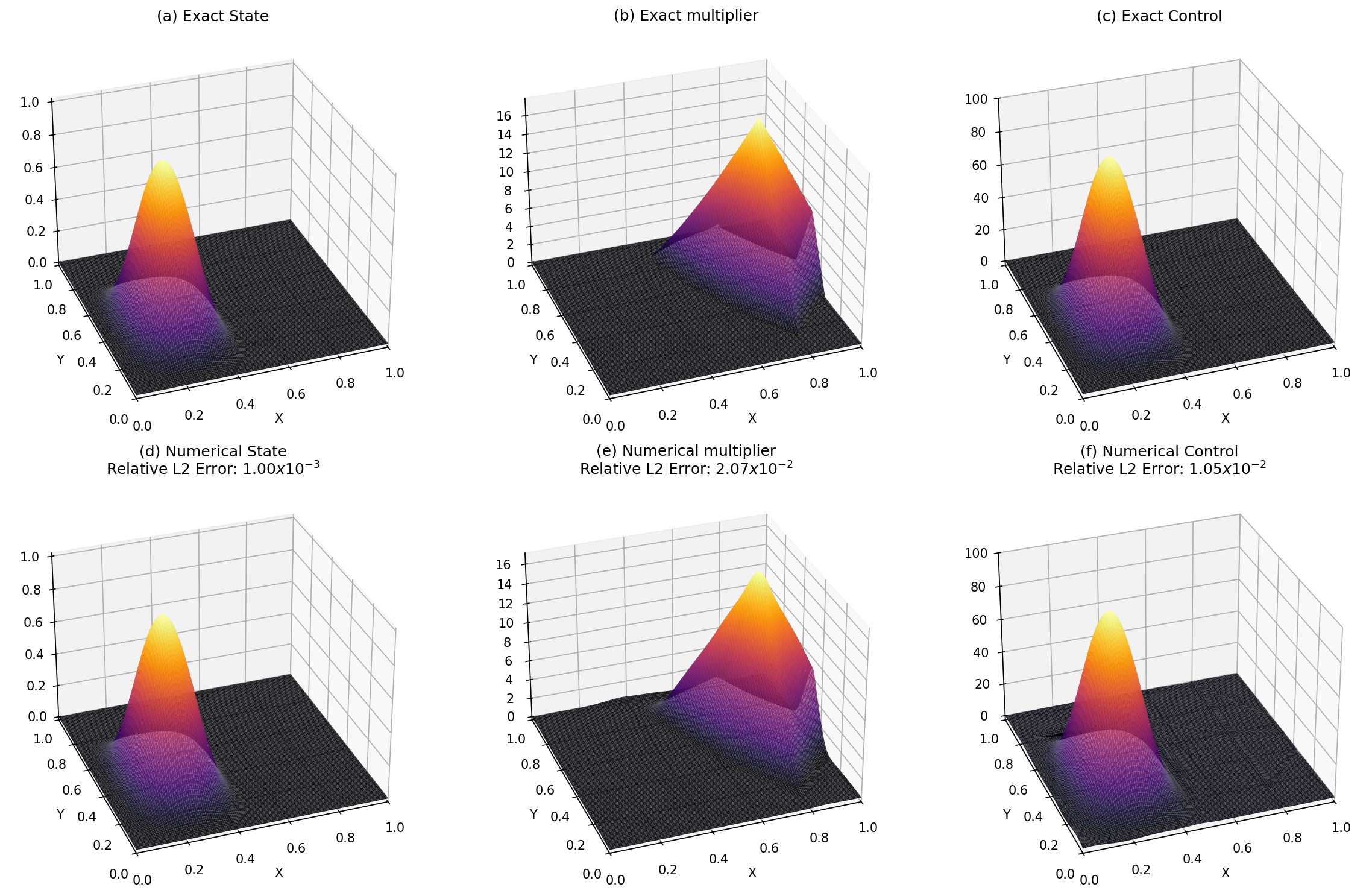}
    \caption{Numerical results for Example \ref{subsec:mpec}. The first row shows the exact optimal state, multiplier, and control; the second row shows the numerical optimal state, multiplier and control.}
    \label{fig:MPEC}
\end{figure}

\subsection{Navier-Stokes optimal control with a high-frequency adjoint}
\label{subsec:navier}

In Section~\ref{subsec:lowregu} we demonstrated the advantage of the exact reduced neural method for problems with low-regularity multipliers. Here we consider another challenging scenario where the adjoint function is dominated by high-frequency modes, which are known to be difficult for neural networks to approximate accurately~\cite{xu2025overview}. Our numerical results indicate that, in this case, the exact reduced neural method is more effective than the corresponding KKT-based neural approach.

We consider the optimal control problem with $\Omega=(0,1)^2$ and $\alpha=0.1$:
\begin{equation}
\text{minimize}\quad J(y,u)=\frac12\|y-y_d\|^2_{L^2(\Omega)}+\frac{\alpha}{2}\|u\|^2_{L^2(\Omega)},
\quad
\text{over }(y,u)\in H^1_0(\Omega)\times L^2(\Omega),
\end{equation}
subject to the steady-state incompressible Navier--Stokes equations
\begin{equation}
    \begin{cases}
    \begin{aligned}
        -\mu \Delta y+(y\cdot\nabla)y+\nabla p &= u +f\quad &\text{in } \Omega,\\
        \text{div } y&=0\quad &\text{in }\Omega,\\
        y&=0\quad &\text{on }\partial\Omega,
        \end{aligned}
    \end{cases}
\end{equation}
where $\mu=0.5$ denotes the reciprocal Reynolds number. We now construct a constrained neural parameterization for the velocity field to enforce both the divergence-free condition and the homogeneous boundary condition. For any divergence-free vector field $y=(y_1,y_2)^T$, there exists a stream function $\psi$ such that
\begin{equation}
y_1=\frac{\partial\psi}{\partial x_2},
\qquad
y_2=-\frac{\partial\psi}{\partial x_1}.
\end{equation}
Since $y=0$ on $\partial\Omega$ and $\Omega=(0,1)^2$, the stream function must be constant on the boundary. Without loss of generality we set $\psi=0$ on $\partial\Omega$. We now let $w(x)=\sin(\pi x_1)\sin(\pi x_2),$ and represent the stream function as $\psi = \phi w$ for some latent function $\phi$. Then we have 
\begin{equation}
    \frac{\partial \psi}{\partial x_2}=\frac{\partial\phi}{\partial x_2}w + \frac{\partial w}{\partial x_2}\phi=0,\quad \frac{\partial \psi}{\partial x_1}=\frac{\partial\phi}{\partial x_1}w + \frac{\partial w}{\partial x_1}\phi=0,\quad \text{on }\partial\Omega.
\end{equation}
Since $w=0$ and $|\frac{\partial w}{\partial x_1}| + |\frac{\partial w}{\partial x_2}|\ne 0$ on $\partial \Omega$, we have $\phi=0$ on $\partial\Omega$. We therefore write $\phi = hw$ for a latent function $h$, which yields $\psi = h w^2$.
Then we introduce a network $\mathcal{N}(\cdot,\theta_y)$ and propose a CNP for $y$:
\begin{equation}
     y_1(\cdot,\theta_y)=\frac{\partial\psi(\cdot,\theta_y)}{\partial x_2},\quad y_2(\cdot,\theta_y)=-\frac{\partial\psi(\cdot,\theta_y)}{\partial x_1}, \quad \text{where }\psi(\cdot,\theta_y) = \mathcal{N}(\cdot,\theta_y)w^2.
\end{equation}
This construction automatically enforces both $\operatorname{div}y=0$ and $y=0$ on $\partial\Omega$. We further introduce a neural network for the pressure, $p(\cdot,\theta_p)=\mathcal N(\cdot,\theta_p)$. The exact reduced neural method then minimizes
\begin{equation}
\text{Loss}(\theta_y,\theta_p)
=
J\big(y(\cdot,\theta_y),u(\cdot,\theta_y,\theta_p)\big),
\end{equation}
where the control is obtained directly from the PDE relation
\begin{equation}
u(\cdot,\theta_y,\theta_p)
=
-\mu \Delta y(\cdot,\theta_y)
+
(y(\cdot,\theta_y)\cdot\nabla)y(\cdot,\theta_y)
+
\nabla p(\cdot,\theta_p)
-
f .
\end{equation}
For comparison we also consider a neural solver based on the KKT system. The adjoint equation reads
\begin{equation}
\begin{cases}
\begin{aligned}
     -\mu \Delta\lambda - (y\cdot\nabla)\lambda+(\nabla y )^\top\lambda + \nabla\nu &=y-y_d\quad &\text{in }\Omega,\\
    \text{div }\lambda&=0 \quad &\text{in } \Omega,\\
    \lambda &= 0,\quad &\text{on }\partial\Omega.
\end{aligned}
\end{cases}
\end{equation}
We then introduce another network $\mathcal{N}(\cdot,\theta_\lambda)$ and employ the following CNP for $\lambda$:
\begin{equation}
\lambda_1(\cdot,\theta_\lambda)=\frac{\partial\varphi(\cdot,\theta_\lambda)}{\partial x_2},\quad \lambda_2(\cdot,\theta_\lambda)=-\frac{\partial\varphi(\cdot,\theta_\lambda)}{\partial x_1},\quad \text{where }\varphi(\cdot,\theta_\lambda) = \mathcal{N}(\cdot,\theta_\lambda)w^2,
\end{equation}
as well as a network $\nu(\cdot,\theta_\nu)=\mathcal{N}(\cdot,\theta_\nu)$ for $\nu$. 
Let $\Theta=\{\theta_y,\theta_p,\theta_\lambda,\theta_\nu\}$, and use the optimality condition $u=-\alpha^{-1}\lambda$ to reduce the control. The loss function based on the KKT system becomes
\begin{equation*}
\begin{aligned}
\text{Loss}(\Theta)
=&\;
\|-\mu \Delta y(\cdot,\theta_y)
+
(y(\cdot,\theta_y)\cdot\nabla)y(\cdot,\theta_y)
+
\nabla p(\cdot,\theta_p)
+
\alpha^{-1}\lambda(\cdot,\theta_\lambda)
-
f
\|^2_{L^2(\Omega)}
\\
&+
\|-\mu \Delta\lambda(\cdot,\theta_\lambda)
-
(y(\cdot,\theta_y)\cdot\nabla)\lambda(\cdot,\theta_\lambda)
+
(\nabla y(\cdot,\theta_y))^T\lambda(\cdot,\theta_\lambda)
+
\nabla\nu(\cdot,\theta_\nu)
-
y(\cdot,\theta_y)
+
y_d
\|^2_{L^2(\Omega)} .
\end{aligned}
\end{equation*}
The exact solutions are chosen as
\begin{equation*}
y^*
=
e^{-0.05\mu}
\begin{pmatrix}
\sin^2(\pi x_1)\sin(\pi x_2)\cos(\pi x_2)\\
-\sin^2(\pi x_2)\sin(\pi x_1)\cos(\pi x_1)
\end{pmatrix},\quad\lambda^*
=
\left(e^{-0.05\mu}-e^{-\mu}\right)
\begin{pmatrix}
\sin^2(8\pi x_1)\sin(8\pi x_2)\cos(8\pi x_2)\\
-\sin^2(8\pi x_2)\sin(8\pi x_1)\cos(8\pi x_1)
\end{pmatrix},
\end{equation*}
and $u^*=-\alpha^{-1}\lambda^*$. The pressures $p^*$ and $\nu^*$ are set to zero. The functions $f$ and $y_d$ are then determined from the optimality system.
% \begin{equation}
% \begin{aligned}
% f&=-\mu \Delta y^*+(y^*\cdot\nabla)y^*+\nabla p^*-u^*,\\
% y_d&=y^*+\mu \Delta\lambda^*+(y^*\cdot\nabla)\lambda^*-(\nabla y^*)^T\lambda^*-\nabla\nu^* .
% \end{aligned}
% \end{equation}

Both methods are trained for 20000 epochs. The learning rate is initialized at $10^{-3}$ and reduced by a factor of $0.8$ every 500 epochs. At each epoch the loss functions are evaluated on a uniform grid of $(50+N)^2$ points, where $N$ is drawn uniformly from $\{0,\ldots,50\}$. Fig.~\ref{fig:NS} shows the exact control, the numerical controls obtained by both methods, and the corresponding relative $L^2$ error histories. Both methods produce accurate approximations; however, the exact reduced neural method achieves higher accuracy with a final relative error of $6.18\times10^{-3}$, which is less than half of the error obtained by the KKT-based method. Moreover, the exact reduced neural method converges significantly faster: after approximately 2500 epochs the error is already reduced to about $1.5\times10^{-2}$. In addition, the KKT-based approach requires more training time and memory per epoch due to the additional networks and derivatives involved in the adjoint equation.
\begin{figure}[htbp]
     \centering
    \begin{subfigure}[b]{0.6\textwidth}
        \centering
        \includegraphics[width=\textwidth]{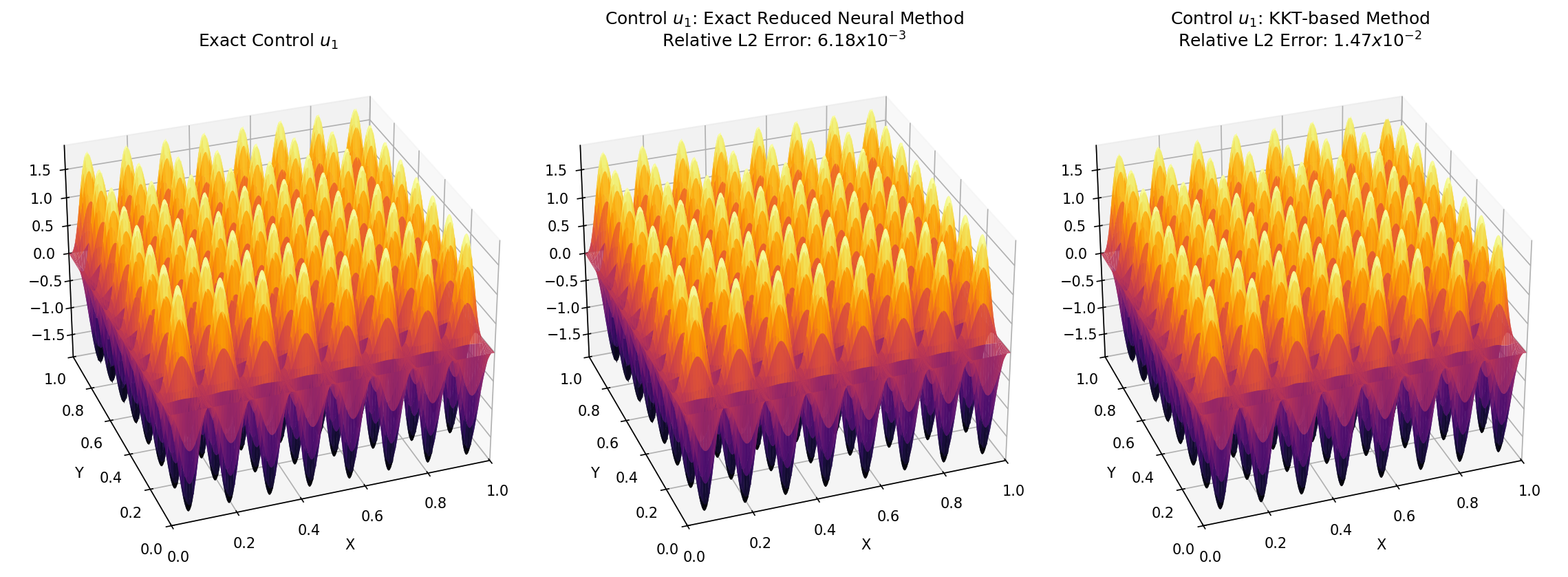}
        \caption{Exact and numerical controls for $u_1$.}
    \end{subfigure}
   % \hfill 
      \hspace{0.5cm}
    \begin{subfigure}[b]{0.25\textwidth}
        \centering
        \includegraphics[width=\textwidth]{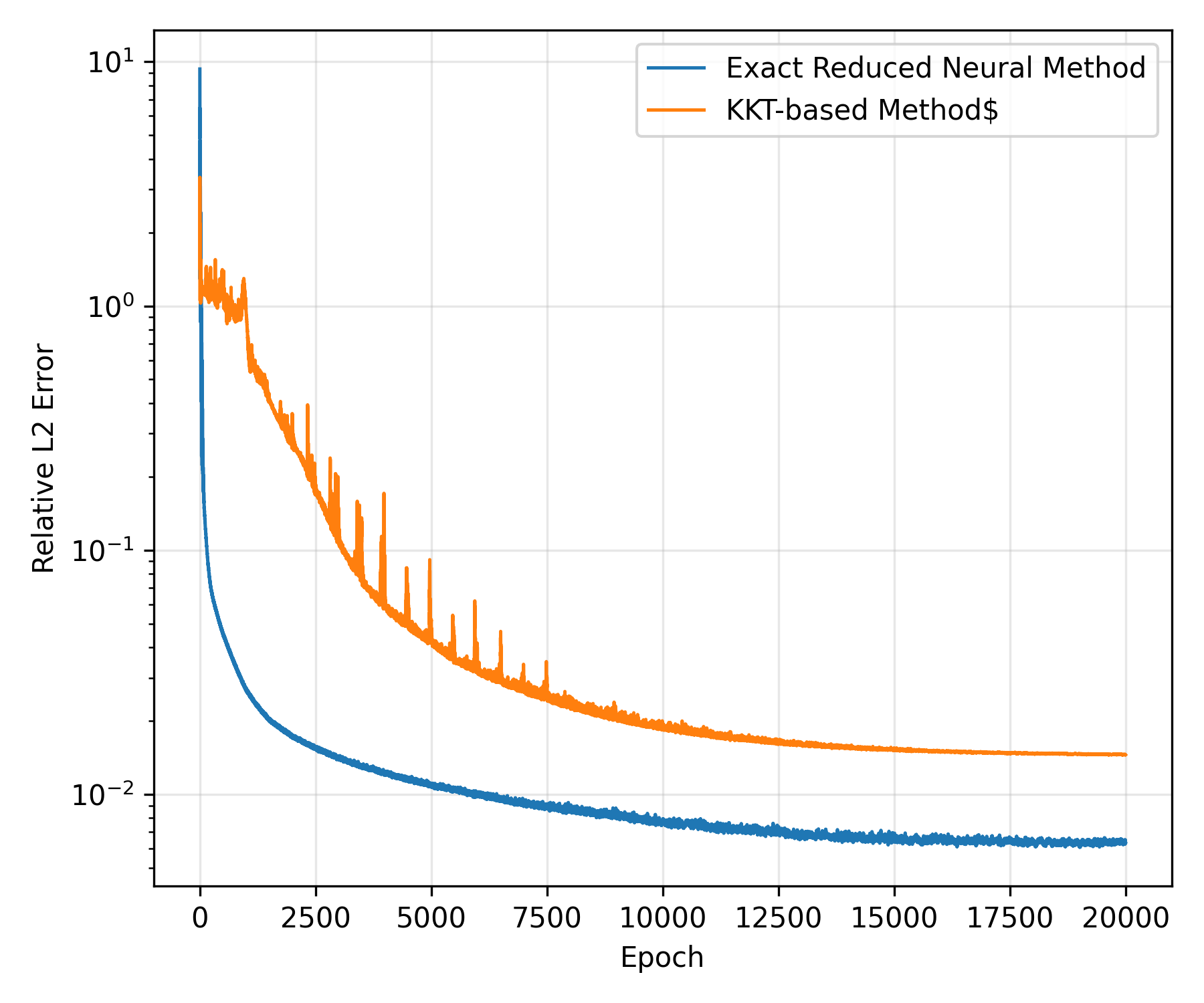}
        \caption{Error histories.}
    \end{subfigure}
     \caption{Numerical results for Example \ref{subsec:navier}. Exact control and numerical controls computed by the exact reduced neural method and the KKT-based neural method, as well as their error histories during training.}
    \label{fig:NS}
\end{figure}

\subsection{Optimal control problem with low-regularity state in an L-shaped domain}
We consider the optimal control problem
\begin{align*}
    &\text{minimize}\quad J(y,u)=\|y-y_d\|_{L^2(\Omega)}^2 + \frac{\alpha}{2}\|u\| _{L^2(\Omega)}^2\qquad \text{over }(y,u)\in H^1_0(\Omega)\times L^2(\Omega),\\
&\text{subject to}\quad
%\begin{equation}
-\Delta y = u+f  \text{ in } \Omega, 
\qquad y = 0  \text{ on } \partial \Omega.
\end{align*}
We choose $\alpha=0.5$ and $\Omega$ to be the L-shaped domain $(-1,1)^2\setminus ([0,1)\times(-1,0])$. The singular function at the corner $(0,0)$ is $s=r^{\frac{2}{3}}\sin(\frac{2\theta}{3})$. We set $\rho=1$ and $L=1/2$ in \eqref{equa:rho}. The optimal state and control are given as $y^* = \mathcal{R}+\eta_\rho s$ and $u^*=50y^*$, with $\mathcal{R}$ given as
\begin{equation}
    \mathcal{R} = \begin{cases}
    \begin{aligned}
         &\sin(2\pi x_1)(\frac{1}{2}x_2^2+x_2)(x_2^2-1),\quad &-1\le x_2\le 0,\\
        &\sin(2\pi x_1)(-\frac{1}{2}x_2^2+x_2)(x_2^2-1),\quad &0< x_2\le 1.
    \end{aligned}
    \end{cases}
\end{equation}
Based on the optimality system, the function $f$ and $y_d$ are determined as
%\begin{equation}
 $f=-\Delta y ^*-u^*,\quad y_d = -\alpha \Delta u ^* + y^*$.
%\end{equation}
A smooth boundary-lifting function for this domain is not directly available. Instead, we introduce a network $\mathcal{N}(\cdot,\theta_w)$, and let $w_{\theta_w}=\mathcal{N}(\cdot,\theta_w)(1-x_1^2)(1-x_2^2)$. Then we learn a smooth boundary-lifting function by minimizing the following loss function:
\begin{equation}
    \text{Loss}(\theta_w)=\|w_{\theta_w}-\operatorname{dist}(\cdot,\partial\Omega)\|_{L^2(\Omega)}^2 + \beta\|w_{\theta_w}\|_{L^2(\Gamma)}^2,
\end{equation}
where 
\begin{equation}
   \operatorname{dist}(x,\partial\Omega):=\min_{x_b\in \partial\Omega}\|x-x_b\|_2 ,\qquad \Gamma = \big(\{0\}\times (-1,0)\big) \cup \big((0,1)\times \{0\}\big).
\end{equation}
The smooth property of $w_{\theta_w}$ can be guaranteed by choosing a smooth activation function, and the boundary condition of $w_{\theta_w}$ on $\partial\Omega\setminus\Gamma$ is satisfied exactly by construction. The first term in the loss function enforces the positivity of $w_{\theta_w}$ in $\Omega$, while the second term enforces the boundary condition on $\Gamma$. Once the training is finished, we fix $\theta_w$ and employ $w_{\theta_w}$ as the smooth boundary-lifting
function.

To implement the singularity-enriched exact reduced neural method, we introduce a network $\mathcal{N}(\cdot,\theta_R)$ and a trainable parameter $\gamma$, then approximate the state and control as 
\begin{equation}
y(\cdot,\hat{\theta})=\mathcal{N}(\cdot,\theta_R)w_{\theta_w} + \gamma(\eta_\rho s),\quad u(\cdot,\hat{\theta})=-\Delta (\mathcal{N}(\cdot,\theta_R)w_{\theta_w})-\gamma\Delta (\eta_\rho s)-f,\quad \hat{\theta}=\{\theta_R,\gamma\}.
\end{equation}
The loss function then reads
\begin{equation}\label{equa:loss_SEERNM}
\text{Loss}(\hat{\theta})=J(y(\cdot,\hat{\theta}), u(\cdot,\hat{\theta})).
\end{equation}
As a comparison, we directly apply the exact reduced neural method without singularity decomposition, i.e., introducing a network $\mathcal{N}(\cdot,\theta_y)$  and approximating the state and control as 
\begin{equation}
y(\cdot,\theta_y)=\mathcal{N}(\cdot,\theta_y)w_{\theta_w},\quad u(\cdot,\theta_y)=-\Delta (\mathcal{N}(\cdot,\theta_y)w_{\theta_w})-f.
\end{equation}
The loss function then reads
\begin{equation}\label{equa:loss_ERNM}
\text{Loss}(\theta_y)=J(y(\cdot,\theta_y), u(\cdot,\theta_y)\big).
\end{equation}
The boundary-lifting function is trained for 10000 epochs with $\beta=10^3$. Both methods are trained for 20000 epochs to minimize the loss functions \eqref{equa:loss_SEERNM} and \eqref{equa:loss_ERNM}, respectively. The learning rates are initialized to $10^{-3}$ for the network parameters and $10^{-2}$ for $\gamma$, and are reduced by a factor of $0.8$ every 500 epochs. The numerical states and controls are presented in Fig.~\ref{fig:singular}. We observe that the exact reduced neural method without singularity decomposition exhibits large errors near the singular corner for both state and control, with relative $L^2$ errors of $2.81\times~10^{-2}$ and $4.32\times 10^{-2}$, respectively. Owing to the singularity decomposition, the singularity-enriched exact reduced neural method can significantly improve the accuracy, with relative $L^2$ errors of $4.02\times 10^{-3}$ and $1.06\times 10^{-2}$ for the state and control, respectively. The learned $\gamma$ is $0.9997$. 
\begin{figure}
    \centering
\includegraphics[width=0.85\linewidth]{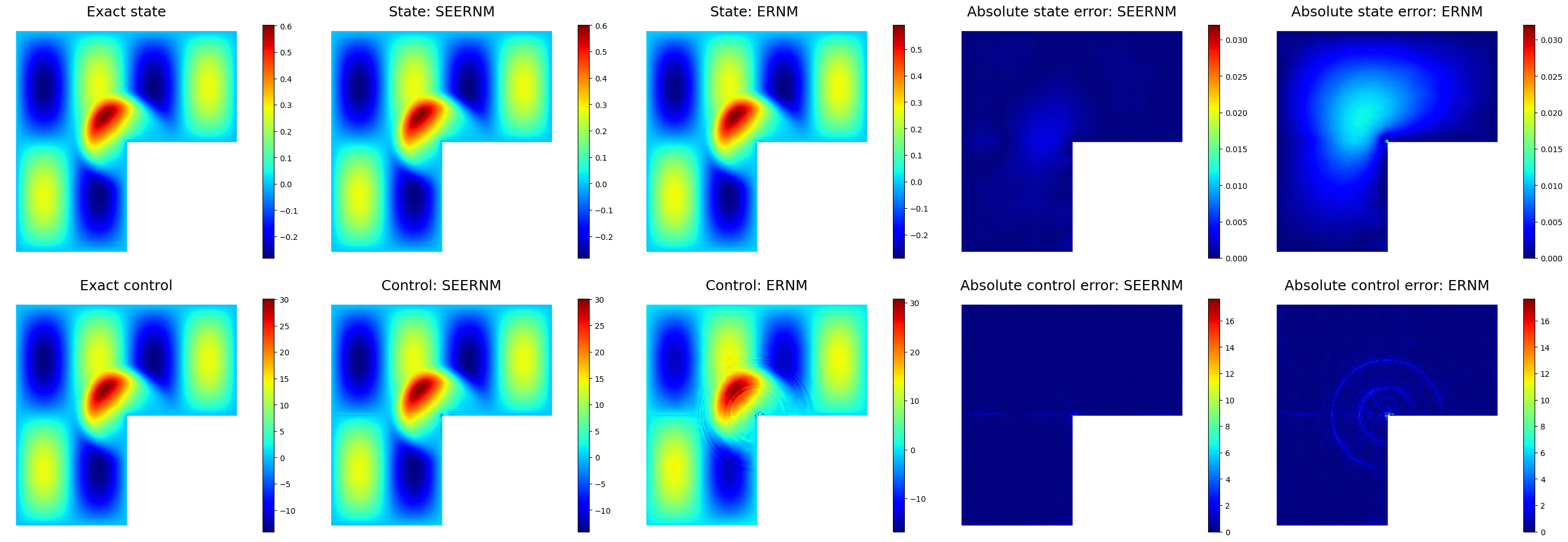}
    \caption{Numerical states and controls and their absolute errors obtained by the singularity-enriched exact reduced neural method (SEERNM) and the exact reduced neural method without singularity decomposition (ERNM).}
    \label{fig:singular}
\end{figure}

\subsection{A hybrid inverse problem}
\label{subsec:IP}

In this example we consider a hybrid inverse problem governed by the second-order elliptic equation
\begin{equation}
\begin{cases}
-\nabla\cdot(\gamma\nabla u)+\sigma(x)u(x)=f,
& \text{in }\Omega,\\
u=h,
& \text{on }\partial\Omega,
\end{cases}
\end{equation}
where $\Omega=(-1,1)^2$, $\gamma\equiv0.1$, and the Dirichlet data $h$ are assumed to be known exactly. The goal is to simultaneously recover the absorption coefficient $\sigma$ and the source term $f$ from the internal measurement $H=\sigma u|_{\Omega}$ together with the Neumann boundary data $g=\partial u/\partial n|_{\partial\Omega}$. This inverse problem is motivated by applications in biomedical imaging, such as photoacoustic imaging, as well as in industrial nondestructive testing. In these settings, the boundary data $g$ correspond to surface flux measurements, while the internal data $H$, though idealized, are commonly used to model hybrid or multi-modal imaging scenarios in which quantities involving both the material property and the field variable are accessible.

For the numerical experiment we use the noisy measurements
\[
g_i^\delta
=
\frac{\partial u(x_i^b)}{\partial n}(1+\delta_i^b),
\qquad
H_i^\delta
=
\sigma(x_i)u(x_i)(1+\delta_i),
\]
where the boundary points $\{x_i^b\}$ are equispaced along $\partial\Omega$ with $128$ points on each side, while the interior  points $\{x_i\}\subset\Omega$ form a $128\times128$ uniform grid in $\Omega$. The noise variables $\{\delta_i\}$ and $\{\delta_i^b\}$ are independent samples from $N(0,0.005^2)$. We also impose a prior information $1\le\sigma\le2$ in $\Omega$ and $\sigma=1$ on $\partial\Omega$. The Dirichlet boundary data are set to $h\equiv1$. The exact solutions used to generate the synthetic data are
\begin{equation}
    \begin{cases}
        u^*(x)=e^{\frac{x_1^2+x_2^2}{\pi}} \sin(\pi x_1)\sin(\pi x_2)+1,\\
        \sigma^*(x)=\max\!\left(\min\!\left(0.8 + 2\sin(\pi x_1)\sin(\pi x_2),\,2\right),\,1\right),\\
        f^*(x)=-\gamma \Delta u^*(x)+\sigma^*(x) u^*(x).
    \end{cases}
\end{equation}
To solve the inverse problem we employ the exact reduced neural method. Two neural networks, $\mathcal N(\cdot;\theta_u)$ and $\mathcal N(\cdot;\theta_\sigma)$, are introduced. The following CNPs are used to enforce the boundary condition for $u$ and to encode the prior bounds on $\sigma$, respectively.
\begin{equation}
    u(\cdot,\theta_u)=\mathcal N(\cdot,\theta_u)w+1,
\qquad
\sigma(\cdot,\theta_\sigma)
=
1+\sin^2\!\big(\mathcal N(\cdot,\theta_\sigma)w\big),
\end{equation}
where \(w(x)=(1-x_1^2)(1-x_2^2)\). The networks are trained by minimizing
\begin{equation}
\text{Loss}(\theta_u,\theta_\sigma)
=
\frac1M\sum_{i=1}^M
\left|
\frac{\partial u(x_i^b,\theta_u)}{\partial n}-g_i^\delta
\right|^2
+
\frac1N\sum_{i=1}^N
\left|
\sigma(x_i,\theta_\sigma)u(x_i,\theta_u)-H_i^\delta
\right|^2 .
\end{equation}
We note that no explicit regularization term is included in the loss function, relying instead on the inherent regularization property of neural networks, often referred to as the low-frequency principle \cite{xu2025overview}. Moreover, second-order derivatives of $u(\cdot,\theta_u)$ are not required during the training process; they are only evaluated after training to compute the recovered source term
\[
f_{\theta} := -\gamma \Delta u(\cdot,\theta_u) + \sigma(\cdot,\theta_\sigma)u(\cdot,\theta_u).
\]
The networks are trained for 20000 epochs. The learning rate is initialized at $10^{-3}$ and reduced by a factor of $0.8$ every 500 epochs. The reconstructions of $\sigma(\cdot,\theta_\sigma)$ and $f_{\theta}$ are presented in Fig.~\ref{fig:hybridIP}. The results show that both $\sigma$ and $f$ are recovered accurately, with relative $L^2$ errors of $1.01\times 10^{-2}$ and $2.10\times10^{-2}$, respectively.
\begin{figure}
    \centering
    \includegraphics[width=0.85\linewidth]{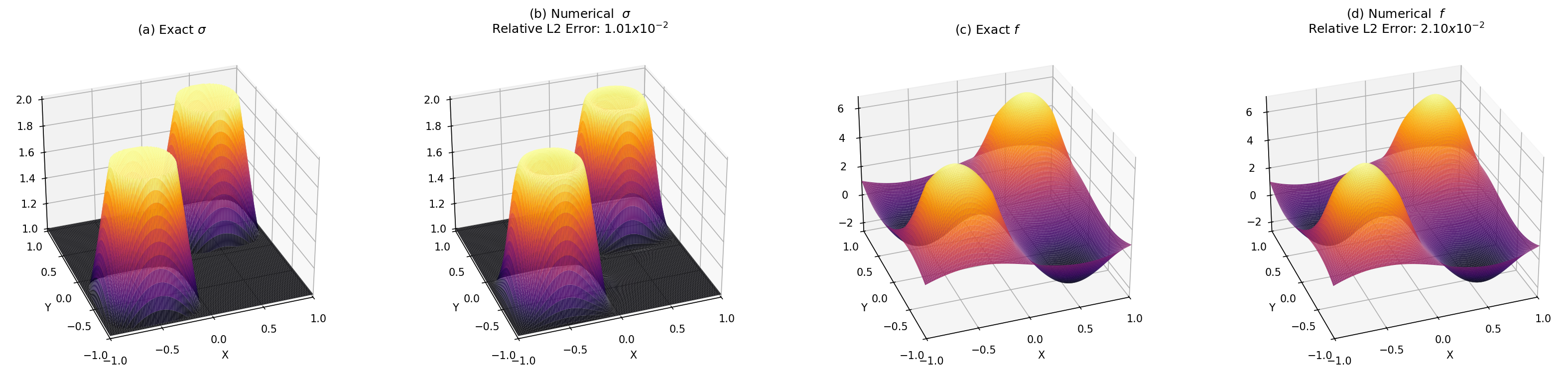}
    \caption{Numerical results for Example \ref{subsec:IP}. Exact and reconstructed absorption coefficient and source term.}
    \label{fig:hybridIP}
\end{figure}

\section{Conclusions}
\label{sec:conclusion}
In this paper we have introduced several constrained neural parameterization (CNP) schemes for optimization problems in function spaces. The central idea is to construct smooth neural parameterizations whose image lies entirely inside the admissible set while remaining dense in that set. This transforms the original constrained optimization problem into an unconstrained optimization problem in parameter space, thereby eliminating the need for penalty parameters, projection steps, or multiplier updates during training.

We established a geometric decomposition framework for polyhedral constraint sets in Hilbert spaces based on a Minkowski--Weyl type representation. These results enable exact neural parameterizations of polyhedral constraints, covering both cases where the constraint functions are not directly available and those where additional Dirichlet boundary conditions are imposed. We further developed smooth neural constructions for several classes of commonly encountered constraints with and without Dirichlet boundary conditions. 
For PDE-constrained optimization problems, we proposed an exact reduced neural formulation for systems admitting separable structures, which enforces the PDE constraint exactly while avoiding the explicit approximation of adjoint variables and Lagrange multipliers. Numerical experiments on several challenging examples were presented to demonstrate the effectiveness of the proposed methodology. Across these examples, the CNP framework consistently produced accurate solutions while maintaining exact feasibility and stable training behavior, often outperforming penalty-based neural methods and KKT-based neural solvers.

Several directions for future research remain open. 
First, a rigorous convergence analysis of the proposed parameterizations in the context of gradient-based optimization would provide a deeper theoretical understanding and guide the design of more efficient architectures. Second, the exact reduced neural formulation relies on the existence of a separable structure of the PDE system; extending this idea to variational formulations or to problems where the control appears nonlinearly remains an open challenge. Third, extending the framework to more general constraint structures would broaden its applicability. In particular, an interesting direction is the development of CNP schemes for mixed constraint sets that combine different types of constraints, such as polyhedral constraints together with pointwise inequality constraints. Finally, another promising direction is the integration of CNP schemes with operator-learning architectures and neural surrogate models for large-scale optimization and control problems.\\[1ex]

\noindent {\bf Acknowledgement:} MH acknowledges support by the Deutsche Forschungsgemeinschaft (DFG, German Research Foundation) under Germany’s Excellence Strategy – The Berlin Mathematics Research Center
MATH+ (EXC-2046/1, project ID: 390685689).

\bibliography{references}

\end{document}